\documentclass[a4paper,12pt]{amsart}

\usepackage{amsthm,amssymb}
\usepackage{stmaryrd}
\usepackage[T1]{fontenc}
\usepackage[utf8]{inputenc}
\usepackage{enumitem}
\usepackage{mathrsfs}
\usepackage{dsfont}
\usepackage{mathtools}
\usepackage{scalerel,stackengine}
\usepackage{nicefrac}
\usepackage[margin=3cm]{geometry}
\usepackage[numbers, sort&compress]{natbib}
\usepackage{caption}
\usepackage{xspace}

\usepackage[all]{xy}
\numberwithin{equation}{section}

\usepackage[hypertexnames=false]{hyperref}
\hypersetup{
    colorlinks,
    linkcolor={red!90!black},
    citecolor={green!70!black},
    urlcolor={b}
    }

\usepackage{tikz}
\usepackage{tikz-cd}
\usetikzlibrary{calc,hobby}
\usetikzlibrary{positioning}
\usetikzlibrary{backgrounds}
\usetikzlibrary{patterns}

\makeatletter % we
\DeclareSymbolFont{lettersA}{U}{txmia}{m}{it}
\SetSymbolFont{lettersA}{bold}{U}{txmia}{bx}{it}
\DeclareFontSubstitution{U}{txmia}{m}{it}

\DeclareMathSymbol{\m@thbbch@rN}{\mathord}{lettersA}{142}
\DeclareMathSymbol{\m@thbbch@rR}{\mathord}{lettersA}{146}
\long\def\DoLongFutureLet #1#2#3#4{%
   \def\@FutureLetDecide{#1#2\@FutureLetToken
      \def\@FutureLetNext{#3}\else
      \def\@FutureLetNext{#4}\fi\@FutureLetNext}
   \futurelet\@FutureLetToken\@FutureLetDecide}
\def\DoFutureLet #1#2#3#4{\DoLongFutureLet{#1}{#2}{#3}{#4}}
\def\@EachCharacter{\DoFutureLet{\ifx}{\@EndEachCharacter}%
   {\@EachCharacterDone}{\@PickUpTheCharacter}}
\def\m@keCharacter#1{\csname\F@ntPrefix#1\endcsname}
\def\@PickUpTheCharacter#1{\m@keCharacter{#1}\@EachCharacter}
\def\@EachCharacterDone \@EndEachCharacter{}

\DeclareRobustCommand*{\varmathbb}[1]{\gdef\F@ntPrefix{m@thbbch@r}%
  \@EachCharacter #1\@EndEachCharacter}
\makeatother

\newcommand{\ca}[1]{\mathcal{#1}}
\newcommand{\mfr}{\mathfrak}
\newcommand{\ms}[1]{\mathscr{#1}}

\newcommand{\calC}{\ca{C}}
\newcommand{\calF}{\ca{F}}
\newcommand{\calG}{\ca{G}}
\newcommand{\calM}{\ca{M}}

\newcommand\frB{\mfr{B}}
\newcommand\frS{\mfr{S}}

\newcommand{\scrK}{\ms{K}}
\newcommand{\scrM}{\ms{M}}
\newcommand{\scrS}{\ms{S}}
\newcommand{\scrZ}{\ms{Z}}

\newcommand{\KS}{\scrK_{\scrS}} %ultracontact from stack
\newcommand{\SK}{\scrS_{\scrK}} %stack from ultracontact
\newcommand{\Smin}{\scrS_{\min}}
\newcommand{\Smax}{\scrS_{\max}}
\newcommand{\Kmin}{\ensuremath{{\scrK_{\min}}}\xspace}
\newcommand{\Kmax}{\ensuremath{{\scrK_{\max}}}\xspace}

\newcommand{\suchthat}{\,\middle\vert\,}

\newcommand{\zero}{\mathbf{0}} %zero of an algebra
\newcommand{\one}{\mathbf{1}} %unity of an algebra
\newcommand{\iffslim}{\longleftrightarrow}
\newcommand{\defeq}{\coloneqq} %requires mathtools package

\newcommand{\iffdef}{\;\mathrel{\mathord{:}\mathord{\longleftrightarrow}}\;}
\newcommand\dftt{\mathtt{df}\,}
\newcommand{\powerne}[1]{\ensuremath{2^{#1}_{+}}}
\newcommand\klam[1]{\left\langle#1\right\rangle}

\DeclareMathOperator{\Int}{Int}
\DeclareMathOperator{\RC}{RC}

\DeclareMathOperator{\con}{\mathsf{C}}
\newcommand{\Kon}{\con_{\mathscr{K}}}
\DeclareMathOperator{\Overl}{\mathsf{O}}

\newcommand\mathbackslash{\raisebox{.4pt}{\texttt{/}}}
\def\notcon{% the separation relation
  \renewcommand\stacktype{L}\mathrel{\ensurestackMath{%
  \ThisStyle{\stackon[0pt]{\SavedStyle\con}{\SavedStyle\mathbackslash}}}}%
}

\DeclareMathOperator{\Cl}{Cl}

\DeclareMathOperator{\upop}{\uparrow} %upper arrow operator
\DeclareMathOperator{\Upop}{\Uparrow} %upper arrow operator for support
\DeclareMathOperator{\downop}{\downarrow} %lower arrow operator

\newcommand{\Nat}{\omega} %the set of natural numbers
\newcommand{\Real}{\varmathbb{R}} %the set of reals

\DeclareMathOperator{\Fil}{\mathsf{Fil}} %the family of filters
\DeclareMathOperator{\Filp}{\mathsf{Fil^{\star}}} %the family of proper filters
\DeclareMathOperator{\Id}{\mathsf{Id}} %the family of ideals
\DeclareMathOperator{\Idp}{\mathsf{Id}^{\star}} %the family of proper ideals
\DeclareMathOperator{\Ult}{\mathsf{Ult}} %all ultrafilters of the algebra
\DeclareMathOperator{\Grill}{\mathsf{Gr}} %the family of grills
\DeclareMathOperator{\Stack}{\mathsf{Stack}}
\DeclareMathOperator{\UC}{\mathsf{UC}} %the collection of all ultracontacts of a BA
\DeclareMathOperator{\STS}{\mathsf{SS}} %the collection of all stack systems of a BA
\DeclareMathOperator{\SC}{\mathsf{S}} %the collection of all simplicial complexes

\DeclareMathOperator{\FC}{FC}
\DeclareMathOperator{\Atom}{At}
\newcommand{\Bne}{B^{+}} %non-zero elements of the algebra B

\newcommand{\tand}{\text{ and }}
\newcommand{\qtand}{\quad\text{and}\quad}

\newcommand{\set}[1]{\ensuremath{\{#1\}}}

\newcommand{\ua}[1]{\ensuremath{\mathop{\uparrow}#1}}
\newcommand{\da}[1]{\ensuremath{\mathop{\downarrow}#1}}

\newcommand{\QED}{\hfill$\dashv$}

\newtheorem{theorem}{Theorem}[section]
\newtheorem{lemma}[theorem]{Lemma}
\newtheorem{proposition}[theorem]{Proposition}
\newtheorem{corollary}[theorem]{Corollary}

\newtheoremstyle{mytheoremstyle} % name
    {1em plus .2em minus .1em}                    % Space above
    {1em plus .2em minus .1em}                    % Space below
    {\rmfamily}                   % Body font
    {}                           % Indent amount
    {\bfseries}                   % Theorem head font
    {.}                          % Punctuation after theorem head
    {.5em}                       % Space after theorem head
    {}  % Theorem head spec (can be left empty, meaning ‘normal’)

\theoremstyle{mytheoremstyle}

\newtheorem{definition}[theorem]{Definition}
\newtheorem{example}[theorem]{Example}
\newtheorem{remark}[theorem]{Remark}
\title{ultracontact Algebras and Stack Systems}

\author[]{Luca Carai, Ivo D\"untsch, Rafa\l\ Gruszczy\'nski, Anna Laura Suarez}

\date{}

\address{Luca Carai\\
Institute of Computer Science\\
Czech Academy of Sciences\\
Prague\\
Czech Republic.
\textsc{Orcid:} 0000-0001-9545-2365
}

\email{carai@cs.cas.cz}

\urladdr{https://lucacarai.github.io/}

\address{Ivo D\"untsch\\
Department of Computer Science\\
Brock University\\	
St Ca\-tha\-ri\-nes, Ontario\\
Canada.
\textsc{Orcid:} 0000-0001-8907-2382}

\email{duentsch@brocku.ca}
\urladdr{https://www.cosc.brocku.ca/~duentsch/}

\address{Rafa\l\ Gruszczy\'nski\\
Department of Logic\\
Institute of Philosophy\\
Nicolaus Copernicus University in Toru\'n\\
Poland.
\textsc{Orcid:} 0000-0002-3379-0577%
}
\email{gruszka@umk.pl}

\address{Anna Laura Suarez,\\
Department of Mathematics and Applied Mathematics\\
University of the Western Cape\\
Cape Town\\
South Africa.
\textsc{Orcid:} 0000-0001-5878-415X%
}

\email{annalaurasuarez993@gmail.com}

\begin{document}

\begin{abstract}

We study the class of structures that, in a way, generalize various approaches to the contact relation on Boolean algebras.

\smallskip


\smallskip


\end{abstract}

\maketitle

\section{Introduction}

Boolean contact algebras (BCAs) were invented to study the phenomenon of nearness of regions in an abstract setting \citep{Stell-BCAANATRCC,Duntsch-Winter-CBCA,Dimov-et-al-CARBTSPA1}. Standardly, they are obtained by expanding the language of Boolean algebras with a~binary predicate `$\con$' of \emph{contact}. The expression `$x\con y$' describes a situation in which---in some sense---region $x$ is in contact with region $y$. In a point-based setting, this holds when $x$ and $y$ share at least one point.

The canonical BCAs are subalgebras of the complete algebras of regular closed subsets of topological spaces. In this context, $x\con y$ if $x\cap y$ (the set-theoretic intersection) is non-empty. If $X$ is a topological space, $\RC(X)$ is its complete Boolean algebra of regular closed sets, $S$ is a subalgebra of $\RC(X)$, and $x,y\in S$, then $\cap$ usually is not an operation of $S$, but of the power set algebra of $X$. So the nearness of $x$ and $y$ is expressed by means of an operation from beyond $S$. In the abstract setting, $\con$ is assumed as primitive, and the standard axioms for it are listed in Section~\ref{sec:C-from-UC}.

In this work, we aim to study an expansion of the notion of contact from pairs of regions to arbitrary collections. For example, we can easily imagine a ternary relation $\con_3$ that holds among $x$, $y$, and $z$ if all three regions are in contact in the sense that there is a~location in space common to all of them. In the point-based environment, the intuition reduces to the existence of at least one point in $x\cap y\cap z$.
The ternary relation $\con_3$ allows us to distinguish between the two configurations shown in Figure~\ref{fig:three regions}, which are indistinguishable from the point of view of the binary contact, since each pair of regions among $x,y,z$ shares a point.

\begin{figure}[ht]
    \centering
\begin{tikzpicture}
  \def\r{2}
  \fill[pattern=north east lines]
    (0,0) -- (90:\r) arc[start angle=90, end angle=210, radius=\r] -- cycle;
  \fill[pattern=crosshatch]
    (0,0) -- (210:\r) arc[start angle=210, end angle=330, radius=\r] -- cycle;
  \fill[pattern=dots]
    (0,0) -- (330:\r) arc[start angle=330, end angle=450, radius=\r] -- cycle;
  \draw (0,0) circle (\r);
  \draw (0,0) -- (90:\r);
  \draw (0,0) -- (210:\r);
  \draw (0,0) -- (330:\r);
  \node at (270:2.5) {$x$};
  \node at (30:2.5) {$y$};
  \node at (150:2.5) {$z$};
\end{tikzpicture}
\hspace{2.5cm}
\begin{tikzpicture}
  \def\ro{2}
  \def\ri{0.8}
  \fill[pattern=north east lines]
    (90:\ri) -- (90:\ro)
    arc[start angle=90, end angle=210, radius=\ro] --
    (210:\ri)
    arc[start angle=210, end angle=90, radius=\ri] -- cycle;
  \fill[pattern=crosshatch]
    (210:\ri) -- (210:\ro)
    arc[start angle=210, end angle=330, radius=\ro] --
    (330:\ri)
    arc[start angle=330, end angle=210, radius=\ri] -- cycle;
  \fill[pattern=dots]
    (330:\ri) -- (330:\ro)
    arc[start angle=330, end angle=450, radius=\ro] --
    (450:\ri)
    arc[start angle=450, end angle=330, radius=\ri] -- cycle;
  \draw (0,0) circle (\ro);
  \draw (0,0) circle (\ri);
  \draw (90:\ri) -- (90:\ro);
  \draw (210:\ri) -- (210:\ro);
  \draw (330:\ri) -- (330:\ro);
    \node at (270:2.5) {$x$};
  \node at (30:2.5) {$y$};
  \node at (150:2.5) {$z$};
\end{tikzpicture}
    \caption{Configurations that are indistinguishable by binary contact relations}
    \label{fig:three regions}
\end{figure}
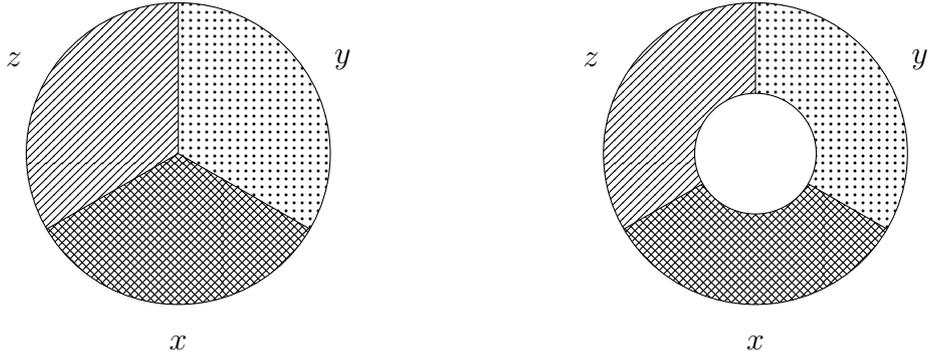
The idea above can be generalized to arbitrary $n$-ary contact relations, as in \cite{Lipparini-HS}. However, we aim at something more far-reaching, and in a~way we want to study in one fell swoop all the possibilities, including infinite ones, by considering a family $\scrK$ of subsets of a Boolean algebra $B$, such that every set $M\in\scrK$ consists of regions sharing a~common location, or having <<a~mutual point of contact>>.
Allowing $\scrK$ to contain infinite subsets of $B$ makes it possible, for example, to express that the family $\{ (0, \nicefrac{1}{n}]\mid n \in \omega \tand n\neq 0\}$ of subsets of the real interval $(0,1]$ lacks a common point of contact; a property that cannot be captured by finitary contact relations.
Thus, in the sequel, we study \emph{ultracontact algebras}, structures $\klam{B,\scrK}$ where $\scrK$ is a family of subsets of $B$.

Throughout the paper, $\klam{B,+, \cdot, -, \zero,\one}$ will denote a non-trivial Boolean algebra, with the operations---respectively---of join, meet, and complement. On $B$ we define, in the standard way, the binary order relation $\leq$. $\Bne$ is the set of all non-zero elements of $B$. If $M\subseteq B$, then $x\in B$ is a \emph{lower bound} of $M$ if for all $y\in M$, $x\leq y$; we write $x\leq M$ in this case. Further, if $x\in B$, then
\begin{equation*}%\tag{$\dftt{\upop}$}
    \upop x\defeq\{y\in B\mid x\leq y\},
\end{equation*}
and for $M\subseteq B$
\[
\upop M\defeq\bigcup_{x\in M}\upop x.
\]
If $B$ and $S$ are Boolean algebras, then we write $S\leq B$ in the case $S$ is a Boolean subalgebra of $B$. $\Atom(B)$ is the collection of all atoms of the algebra $B$, i.e., all its non-zero elements minimal with respect to the Boolean order $\leq$.

For any set $X$, $2^X$ is the power set of $X$, and $\powerne{X}$ is the family of all non-empty subsets of~$X$. $\omega$ is the set of natural numbers, and $\omega^{+}\defeq\omega\setminus\{0\}$.

For any unexplained notions we refer the reader to \cite{kop89} for Boolean algebras and \cite{bd74} for lattice theory.

The paper is organized as follows. In Section~\ref{sec:sfg}, we recall the properties of stacks, filters, and grills. In Section~\ref{sec:support} we introduce and study the relation of \emph{support} that plays a key role in the presentation of the central concept of our work, \emph{ultracontact algebra}, which is introduced in Section~\ref{sec:UCA}. In the same section, we show that the family of ultracontacts on a Boolean algebra is a complete co-Heyting algebra. In Section~\ref{sec:new-UCs} we show how to construct new ultracontacts from old ones, and in Section~\ref{sec:C-from-UC} we demonstrate that the standard contact, and the lesser-known hypercontact are definable in the framework of ultracontact algebras. Section~\ref{sec:fin-UCs-and-SCs} draws a parallel between finite ultracontact algebras and abstract simplicial complexes from algebraic topology. In Section~\ref{sec:stacks} we prove that the notion of \emph{ultracontact algebra} is definitionally equivalent to the purely order-theoretic notion of a \emph{stack system}. Section~\ref{sec:future} lists some problems for future work.

\section{Stacks, filters, grills}\label{sec:sfg}

This section is a refresher on various order-related subsets of Boolean algebras.

\begin{definition}\label{def:Stack}
    A \emph{stack} of $B$ is an upward closed subset $U$ of $B$ w.r.t.\ the Boolean order. The family of stacks of $B$ is denoted by $\Stack(B)$.
\end{definition}

\begin{proposition}[{\citealp[Theorem 8.2]{Harzheim-OS}}]\label{thm:stacklattice}
    $\klam{\Stack(B),\mathord{\cup},\mathord{\cap},\emptyset,B}$ is a~complete lattice, with arbitrary unions and intersections as lattice operations and $\set{\one}$ as its sole atom.
\end{proposition}

\begin{definition}\label{def:filter}
    A subset $F$ of $B$ is called a~\emph{filter} if it satisfies
    \begin{enumerate}[label=(Fil\arabic*),ref=Fil\arabic*]
\addtocounter{enumi}{-1}
    \item $\one\in F$.
    \item\label{Filt1} If $a \in F$ and $a \leq b$, then $b \in F$.
    \item If $a,b \in F$, then $a\cdot b \in F$.
\end{enumerate}
\end{definition}
$F$ is called \emph{proper}, if $\zero\not\in F$. The collection of filters is denoted by $\Fil(B)$, and the collection of proper filters by $\Filp(B)$. A maximal element of $\Filp(B)$ with respect to set inclusion is called an \emph{ultrafilter}. We denote the collection of ultrafilters by $\Ult(B)$. The following description of ultrafilters is well known; see, e.g. \cite[Proposition 2.15]{kop89}:
\begin{theorem}\label{thm:uf}
    For every $F\in\Filp(B)$ the following conditions are equivalent:
    \begin{enumerate}[label=(\arabic*)]
        \item $F\in\Ult(B)$.
       \item $F$ is prime, i.e., for any $a,b\in B$, if $a + b\in F$, then $a\in F$ or $b\in F$.
       \end{enumerate}
\end{theorem}
The dual notion of a filter is that of an ideal:
\begin{definition}\label{def:ideal}
    A subset $I$ of $B$ is an \emph{ideal} if:
    \begin{enumerate}[label=(Id\arabic*),ref=Id\arabic*]
    \addtocounter{enumi}{-1}
        \item $\zero\in I$.
        \item If $a\in I$ and $b\leq a$, then $b\in I$.
        \item If $a\in I$ and $b\in I$, then $a+b\in I$.
    \end{enumerate}
   An ideal is called \emph{proper} if $\one\notin I$. We let $\Id(B)$ be the set of all ideals of $B$,  and $\Idp(B)$ be the family of all proper ideals of~$B$.
\end{definition}

\begin{theorem}[{\citealp[Corollary II.9.4]{bd74}}]\label{thm:latticeideal}
The structure $\klam{\Id(B), \bigwedge, \bigvee, \set{\zero}, B}$ is a complete lattice where infimum is intersection and for $\set{I_n: n \in J}$, $\bigvee \set{I_n: n \in J}$ is the intersection of all ideals containing $\bigcup \set{I_n: n \in J}$.
\end{theorem}

Note that the smallest ideal containing $\emptyset \neq M \subseteq B$ is given by
\begin{gather}\label{genid}
\left\{b \in B\suchthat b \leq \bigvee T \text{ for some finite } T \subseteq M\right\}.
\end{gather}

The following concept is of major importance in the sequel:
\begin{definition}[{\citealp{Choquet-SLNDFEDG}}]\label{def:Grill}
A proper subset $G$ of $B$ is called a \emph{grill} if it is a non-empty stack and $a+b \in G$ implies $a \in G$ or $b \in G$. More concretely,  $G$ is a grill if it satisfies
\begin{enumerate}[label=(Gr\arabic*),ref=Gr\arabic*]
\addtocounter{enumi}{-1}
    \item\label{Gr0} $\one\in G$ and $\zero\notin G$.
    \item\label{Gr1} If $a \in G$ and $a \leq b$, then $b \in G$.
    \item\label{Gr2} If $a+b \in G$, then either $a \in G$ or $b \in G$.
\end{enumerate}
\end{definition}
The family of grills of $B$ is denoted by  $\Grill(B)$. There is an intimate connection between grills and ultrafilters:
\begin{theorem}[{\citealp{Thron-PSAG}}]\label{thm:ultragrill}
Every grill of $B$ is a union of ultrafilters and every non-empty union of ultrafilters is a grill.
\end{theorem}
The following extension result is sometimes called the \emph{Grill Lemma}:
\begin{lemma}[Grill Lemma, {\citealp[Theorem 1.1]{vak89a}}]\label{lem:grill}
     If $G$ is a grill and $F$ is a filter of $B$ such that $F\subseteq G$, then there is an ultrafilter $U\subseteq G$ such that $F\subseteq U$.
\end{lemma}

Let
\[
F+G\defeq\{f+ g\mid f\in F\tand g\in G\}.
\]
$F+G$ is the so-called \emph{complex} (or \emph{Minkowski}) sum of $F$ and $G$, and we have that $F+G$ is non-empty iff both $F$ and $G$ are non-empty.

The following property of grills will be used further in the paper:

\begin{proposition}\label{prop:prime-for-big-union}
 Let $G\in\Grill(B)$. For all $F,H\in2^B$, if $F+H\subseteq G$, then $F\subseteq G$ or $H\subseteq G$.
\end{proposition}
\begin{proof}
Suppose $F\nsubseteq G$ and $H\nsubseteq G$, and let $f\in F\setminus G$ and $h\in H\setminus G$. Then, $f+h\in F+H$ by the definition, yet $f+h\notin G$, by~\eqref{Gr2}.
\end{proof}

\begin{definition}
    $x\neq\one$ is a \emph{meet-prime} element of a lattice if $a\cdot b\leq x$ implies that $a\leq x$ or $b\leq x$.
\end{definition}

\begin{lemma}\label{l: grills are prime}
    Grills are precisely the non-empty meet-prime elements of the lattice $\Stack(B)$.
\end{lemma}
\begin{proof}
By the definition, any grill is different from $B$.

    Suppose that $G$ is a grill, and that $U$ and $V$ are stacks such that $U\nsubseteq G$ and $V\nsubseteq G$. Let $u\in U$ and $v\in V$ be such that $u,v\notin G$. As $G$ is a grill, $u+v\notin G$. But $u+v\in U\cap V$, and so $U\cap V\nsubseteq G$.

    For the converse, suppose that $G\in \Stack(B)$ is a meet-prime element and that $u,v\notin G$. Then, $\ua{u},\ua{v}\nsubseteq G$. By primality, $\ua{u}\cap \ua{v}\nsubseteq G$. As $\ua{u}\cap \ua{v}=\ua{(u+v)}$ and $G$ is a stack, it must be the case that $u+v\notin G$.
\end{proof}

Grills are closely related to ideals:
\begin{lemma}[{\citealp{Thron-PSAG}}]\label{lem:id-fil-gr-bijection}
    $G\in\Grill(B)$ if and only if $B\setminus G\in\Idp(B)$, so there is a one-to-one correspondence between grills and proper ideals.
 \end{lemma}
 The bijective correspondence between grills and proper ideals from the lemma above is an anti-isomorphism between $\Grill(B)$ and $\Idp(B)$. Consequently, the meet of a family of grills $\{G_i\mid i \in I\}$ in $\Grill(B)$ is the complement of the join of $\{B\setminus G_i\mid i\in I\}$ in $\Idp(B)$. However, such a meet exists only for those families $\{G_i\mid i\in I\}$, for which there exists an ultrafilter $U\subseteq\bigcap_{i\in I}G_i$. This implies that the family of grills can be made into a semi-lattice with arbitrary unions of non-empty families as suprema and partial infima given by
\begin{gather*}%\label{l: lattice of grills}
\bigwedge_{i\in I}G_i\defeq B\setminus \bigcup \left\{\da{(g_0+\ldots+ g_n)}\suchthat n\in\Nat\tand g_0,\ldots,g_n\notin \bigcap_{i\in I} G_i\right\}
\end{gather*}
or, equivalently, by
    \[
\bigwedge_{i\in I} G_i\defeq\left\{a\in B\suchthat(\forall b_1,\ldots,b_n\in B)\,(a\leq b_1+\ldots+b_n\Rightarrow(\exists k\leqslant n)\, b_k\in\bigcap_{i\in I} G_i)\right\},
    \]
which is a consequence of the correspondence we pointed to above and the fact that for ideals, the supremum is computed as
    \[
\bigvee_{i\in I} J_i=\left\{a\in B\suchthat(\exists b_1,\ldots,b_n\in\bigcup_{i\in I}J_i)\,a\leq b_1+\ldots+b_n\right\}.
    \]

\section{The support relation}\label{sec:support}
Before we formulate the axioms, we define an auxiliary relation of \emph{support} between subsets of a Boolean algebra, and we investigate its elementary properties.

\begin{definition}\label{def:leqUCA}
For $F,G\in 2^{B}$, we define $F\preceq G$ iff for every $g\in G$ there is $f\in F$ such that $f\leq g$. In the case $F\preceq G$ we say that $F$ \emph{supports} $G$ or that $F$ \emph{is a support} of $G$. \QED
\end{definition}

For $B$ a Boolean algebra, and $F\subseteq B$, we define
\begin{equation*}
   \Upop F\defeq\{G\subseteq B\mid F\preceq G\tand G\neq \emptyset\}.
\end{equation*}
We have that for any $a\in B$ and any $M\subseteq\Bne$
\[
M\in\Upop\{a\}\iff a\leq M\tand M\neq\emptyset,
\]
and
\[
\bigcup_{m\in M}\Upop \{m\}\subseteq \Upop M,
\]
but the reverse inclusion does not hold in general, as the example below shows.

\begin{example}
    Let $x,y\in\Bne$ be such that $x\cdot y=\zero$. Then for $M\defeq\{x,y\}$ we have that $M\in\Upop M$, but $M\notin\Upop\{x\}\cup\Upop\{y\}$.\QED
\end{example}

The support relation is reflexive and transitive but not antisymmetric.
\begin{example}\label{ex:prec}
    Consider the interval algebra $B$ over the reals, that is, $B$ consists of $\emptyset,[0,\infty)$  and those subsets of $[0,\infty)$ that are finite unions of left-closed, right-open intervals $[x,y)$ \cite[Example 1.11]{kop89}. Let $F$ be the set of elements of $B$ with rational endpoints, and $G$ the set of elements with irrational endpoints. Then, $F\neq G$ but $F \preceq G$ and $G \preceq F$.\QED
\end{example}

\begin{example}\label{ex:reals}
    For a more topological example, consider $\RC(\Real)$, the Boolean algebra of regular closed subsets of the reals with the standard topology. The sets $F\defeq\{[-\nicefrac{1}{n^2},\nicefrac{1}{n^2}]\mid n\in\Nat^+\}$ and $G\defeq\{[-\nicefrac{1}{n^3},\nicefrac{1}{n^3}]\mid n\in\Nat^+\}$ mutually support each other, but they are clearly different sets.\QED
\end{example}

As the support relation is a pre-order, its symmetric closure
\begin{equation}\tag{$\dftt{\sim}$}\label{df:equiv-refine}
F\sim G\iffdef F\preceq G\tand G\preceq F.
\end{equation}
is an equivalence relation. If $F\sim G$, then we say that $F$ and $G$ are \emph{similar}.

\begin{proposition}\label{prop:support-very-basic}
    If $F,G,H,M$ are subsets of $B$, then:
    \begin{enumerate}[label=(\arabic*),ref=\arabic*]
    \item $F\preceq F+G$.
    \item $F+G\preceq H$ implies that $F\preceq H$ and $G\preceq H$.
    \item $F \preceq H$ and $G \preceq M$ imply that $F + G \preceq H + M$.
    \item $F\preceq G$ and $G\preceq H$ imply that $F\preceq H$.
    \item $(F+G)+H = F+(G+H)$.\pagebreak
    \item $F\preceq G$ iff $G\subseteq\upop F$.
    \item If $F\subseteq G$, then $G\preceq F$.
    \item $F\sim G$ iff $\upop F=\upop G$.
    \item $F\sim\upop F$.
    \item $\Upop F=2^{\upop F}_{+}$.
    \end{enumerate}
\end{proposition}
\begin{proof}
    Routine.
\end{proof}

 For a subset $F$ of an algebra $B$, let $[F]$ be the equivalence class of $F$ with respect to the similarity relation
\[
[F]\defeq F/_{\sim}.
\]

We can now extend the support relation to the equivalence classes in the obvious way
\[
[F]\preceq [G]\iffdef F\preceq G.
\]
The relation is well-defined, that is, if $F'\in[F], G'\in[G]$ and $[F]\preceq [G]$, then $[F']\preceq[G']$, which is a consequence of the transitivity of the support relation.

\begin{proposition}\label{prop:+}
    $\klam{2^{B}/_{\sim},\preceq}$ is a lattice. The join of $[F]$ and $[G]$ is given by $[F+G]$, and the meet by $[F\cup G]$.
\end{proposition}
\begin{proof}
    Let $B$ be a Boolean algebra, and let $F,G\subseteq B$. Every element of $F+G$ is of the form $f+g$ for some $f\in F$, $g\in G$, and so $f\in F$ and $g\in G$ are such that $f,g\leq f+g$. Then, $F,G\preceq F+G$. Suppose that $F,G\preceq H$ for $H\subseteq B$. Let $h\in H$, and let $f\in F$ and $g\in G$ be such that $f,g\leq h$. Then, $f+g\leq h$. So, $F+G\preceq H$. Thus $F+G$ is the join of $F$ and $G$.

    For the meet, it is clear that $F\cup G\preceq F,G$. Suppose $H\preceq F,G$. We show that $H\preceq F\cup G$. W.l.o.g., let $x$ be an element of $F$. Since by assumption $H\preceq F$, in $H$ there is an $h$ such that $h\leq x$. So $H\preceq F\cup G$. It follows that $F\cup G$ is the meet of $F$ and $G$.
\end{proof}

If $x\in B$, then
\[
\set{x} \sim \ua{x}
\]
Indeed, $\set{x} \preceq \ua{x}$ since $x \leq y$ for all $y \in \ua{x}$. Conversely, $\ua{x} \preceq \set{x}$ since $x \in \ua{x}$.

Let $[x]$ be the equivalence class of $\{x\}$ with respect to the similarity relation
\begin{equation*}
    [x]\defeq\{x\}/_{\mathord{\sim}}.
\end{equation*}

\begin{lemma}\label{lem:equiv-of-x}
    If $x\in B$, then
    \[
       [x] =\{M\subseteq B\mid x\in M\tand x\leq M\}.
    \]
\end{lemma}
\begin{proof}
($\subseteq$) Suppose $M\preceq \{x\}$ and $\{x\}\preceq M$. From the first condition, we have that there is an $m_0\in M$ such that $m_0\leq x$. From the second, $x\leq M$, so in particular $x\leq m_0$. Therefore, by the antisymmetry of $\leq$ we obtain $x=m_0$, i.e., $x\in M$.

\smallskip

($\supseteq$) If $x\in M$, then clearly $M\preceq\{x\}$, and if $x$ is a lower bound of $M$, then $\{x\}\preceq M$. Thus $\{x\}$ and $M$ are similar, as desired.
\end{proof}

\section{Ultracontact Algebras}\label{sec:UCA}

We introduce the central concept of our work.
\begin{definition}\label{def:UCA}
An \emph{ultracontact algebra} (UCA for short) is a pair $\klam{B,\scrK}$ such that $B$ is a Boolean algebra and $\scrK\subseteq2^B$ is a~family---which we call \emph{ultracontact} or, briefly, UC---that satisfies the following axioms for all $F,G\subseteq B$:
\begin{enumerate}[label=(K\arabic*),ref=K\arabic*,itemsep=3pt]
\addtocounter{enumi}{-1}
\item \label{K0} $\emptyset\notin\scrK$.
    \item \label{K1} If $\zero\in F$, then $F\notin \scrK$.
    \item \label{K2} If $x\neq\zero$, then $\{x\}\in \scrK$.
    \item \label{K3} If $F\preceq G\neq\emptyset$ and $F\in \scrK$, then $G\in \scrK$.
    \item \label{K4} If $F+ G\in \scrK$, then $F\in \scrK$ or $G\in \scrK$.\QED
\end{enumerate}
\end{definition}

A straightforward induction on $n$ shows that \eqref{K4} can be generalized to
\begin{equation}\tag{K4$'$}\label{K4-general}
    \text{If } F_1+\ldots+F_n\in \scrK, \text{ then }(\exists i\leqslant n)\, F_i\in \scrK,
\end{equation}
where
\[
F_1+\ldots+F_n\defeq\{f_1+\ldots+f_n\mid (\forall i\leqslant n)\,f_i\in F_i\}.
\]

Observe that \eqref{K2} may be equivalently formulated as
\begin{equation}\tag{K2$'$}
    \Atom\left(2^{\Bne}\right)\subseteq\scrK,
\end{equation}
and \eqref{K3} as
\begin{equation}\tag{K3$'$}
\text{If $F\in\scrK$, then $\Upop F\subseteq\scrK$.}
\end{equation}

\begin{proposition}\label{prop:K3-reversed}
    For any ultracontact algebra $\klam{B,\scrK}$, and any non-empty $F,G\in 2^B$, if $F\in\scrK$ or $G\in\scrK$, then $F+G\in\scrK$.
\end{proposition}
\begin{proof}
    Indeed, if $F\neq\emptyset\neq G$, then $F+G\neq\emptyset$. Moreover, $F\preceq F+G$. So, if $F\in\scrK$, then $F+G\in\scrK$. Similarly, $G\in\scrK$ implies $F+G \in \scrK$.
\end{proof}

\begin{lemma}\label{lem:K-down-closed}
    Any ultracontact $\scrK$ is closed under non-empty subsets of elements
    of~$\scrK$. In particular, $\scrK$ must be closed under non-empty intersections.
\end{lemma}
\begin{proof}
    Let $M$ be an element of $\scrK$ and $N$ be a non-empty subset of $M$. Then $M\preceq N$, by Proposition~\ref{prop:support-very-basic}(7). So $N\in\scrK$ by \eqref{K3}.
\end{proof}

We show that every Boolean algebra carries at least two ultracontacts: the smallest $\Kmin$ and the largest $\Kmax$. Let:
\begin{equation*}
    \begin{split}
    \Kmin\defeq{}&\{M\in\powerne{B}\mid (\exists x\in\Bne)\,x\leq M\}\\
    ={}&\textstyle\bigcup\{\Upop\{x\}\mid x\in\Bne\}
    \end{split}
    \end{equation*}
    and
    \begin{equation*}
        \Kmax\defeq\{M\in\powerne{B}\mid\zero\notin M\}.
    \end{equation*}
\begin{theorem}\label{th:smallest-UC}
    $\klam{B,\Kmin}$ is a UCA, and for every ultracontact $\scrK$ on $B$, $\Kmin\subseteq\scrK$.
\end{theorem}
\begin{proof}
    The axioms \eqref{K0}--\eqref{K2} are obvious.

    \smallskip

    For \eqref{K3}, if $F$ supports $G\neq\emptyset$ and $F\in\Kmin$, then $F$ has a non-zero lower bound~$x$. But $x\leq G$ as well, as any lower bound of $F$ must be a lower bound of~$G$.

    \smallskip

    For \eqref{K4}, if neither $F$ nor $G$ has a~non-zero lower bound, then $\bigwedge F=\zero=\bigwedge G$. By \citep[Lemma 1.33]{kop89} it must be the case that $\bigwedge (F+G)=\zero$, i.e., $F+G\notin\Kmin$.

       \smallskip

    For the second part, let $\scrK$ be an arbitrary ultracontact on $B$ and let $M\in\Kmin$. So, $M\neq\emptyset$ and $M$ has a~lower bound $x\in\Bne$. Then $\{x\}$ supports $M$. Further, by \eqref{K2}, $\{x\}\in\scrK$, so \eqref{K3} entails that $M\in\scrK$. Therefore, $\Kmin\subseteq\scrK$, as required.
\end{proof}

\begin{theorem}\label{th:largest-UC}
    $\klam{B,\Kmax}$ is a UCA, and for any ultracontact $\scrK$ on $B$, $\scrK\subseteq\Kmax$.
\end{theorem}
\begin{proof}
    \eqref{K0}--\eqref{K2} are immediate.

    \smallskip

    \eqref{K3} If $F\in\Kmax$, then $\zero\notin F$. Therefore, if $F$ supports $G$, $G$ cannot contain the zero of the algebra. Thus, if $G$ is non-empty it must be an element of $\Kmax$.

    \smallskip

    \eqref{K4} If $F+G\in\Kmax$, then $\{f+g\mid f\in F\tand g\in G\}$ is non-empty and does not contain the zero. So for every $f\in F$ and every $g\in G$ we have $f+g\neq\zero$, and it cannot be the case that $\zero$ is in both $F$ and $G$. So either $F\in\Kmax$ or $G\in\Kmax$.

    \smallskip

   By \eqref{K0} and \eqref{K1}, it is obvious that if $\scrK$ is an ultracontact, then $\scrK\subseteq\Kmax$.
\end{proof}

As a more concrete example, we turn to the particular case of ultracontact in the topological setting. In the same way as the standard contact axioms are motivated by and abstracted from the basic and general properties of the nearness of sets, our axioms are the abstract counterparts of the generalized phenomenon of nearness.

\begin{example}\label{ex:topological-UC}
   Let $X$ be a topological space, $\RC(X)$ be its complete Boolean algebra of regular closed sets\footnote{A set $A$ is regular closed in $X$ if $A$ is equal to the closure of its interior: $\Cl\Int A=A$.}, and let $S\leq\RC(X)$. For families $\calF,\calG\in 2^{S}$ we define
    \[
        \calF+\calG\defeq\{F\cup G\mid F\in\calF\tand G\in\calG\}.
    \]
   Let $\scrK^S\subseteq 2^{S}$ be the family of all non-empty $\calF\in 2^{S}$ such that $\bigcap_{F\in\ca{F}}F\neq\emptyset$, where $\bigcap$ is the set-theoretical intersection (which is the infimum operation of $2^{X}$, but not necessarily the infimum of $\RC(X)$ or $S$). Then, $\scrK^S$ meets the conditions for ultracontact:
\begin{enumerate}[label=(TK\arabic*),ref=TK\arabic*,itemsep=5pt]
\addtocounter{enumi}{-1}
    \item\label{TK0} $\emptyset\notin\scrK^S$.
    \item\label{TK1} If $\emptyset\in\calF$, then $\calF\notin\scrK^S$.
    \item\label{TK2} If $F\neq\emptyset$, then $\{F\}\in\scrK^S$.
    \item\label{TK3} If $\calF\preceq\calG\ne\emptyset$ and $\calF\in\scrK^S$, then $\calG\in\scrK^S$.
    \item\label{TK4} If $\calF+\calG\in\scrK^S$, then $\calF\in\scrK^S$ or $\calG\in\scrK^S$.
\end{enumerate}
Here, in $\scrK^S$ we want all non-empty collections of regular closed sets that have non-empty intersections, i.e., non-empty infima in $2^X$, but not necessarily in $S$, or even $\RC(X)$. So this time, the ultracontact between a family of regions is witnessed by objects from beyond the algebra.\QED
\end{example}

For a Boolean algebra $B$, we define $\UC(B)$ as the poset of all ultracontacts on $B$, ordered by set inclusion.

\begin{theorem}\label{th:lattice-of-UCAs}
$\UC(B)$ is a complete lattice, where joins are given by set-theoretical unions.
\end{theorem}
\begin{proof}
   It is easy to see that the union of any non-empty family of ultracontacts must be an~ultracontact. In consequence, the poset of all ultracontacts is a join-complete semilattice, and it has the smallest element by Theorem \ref{th:smallest-UC}. Therefore, it is a complete lattice by \cite[Theorem IV.2]{Birkhoff-LT}.
\end{proof}

Below, we explicitly characterize the infimum operation in the lattice $\UC(B)$.

\begin{theorem}\label{th:meet-ultracontact}
If $\calF$ is a non-empty family of ultracontacts on $B$, then the meet $\bigwedge \calF$ in the lattice $\UC(B)$ is

\[
\scrK\defeq\left\{ F \in 2^{B}_+ \suchthat (\forall\, G_1, \ldots, G_n \in 2^{B}_+)\left(F \preceq G_1+ \ldots + G_n \Rightarrow (\exists i \leqslant n)G_i \in \bigcap \calF\right)\right\}.
\]
\end{theorem}
\begin{proof}
Let $\scrK \subseteq 2^B$ be as above.
We first show that $\scrK \subseteq \bigcap \calF$. Consider $F \in \scrK$. Since $F \preceq F$, the definition of $\scrK$ yields $F \in \bigcap \calF$. Thus, $\scrK \subseteq \bigcap \calF$. We now prove that $\scrK$ is an ultracontact.

\smallskip

\eqref{K0} We have that $\emptyset \notin \scrK$ by the definition of $\scrK$.

\smallskip

\eqref{K1} Consider $F \in 2^{B}$ such that $\zero \in F$. Suppose that $F \in \scrK$. Since $\scrK \subseteq \bigcap \calF$ and $\calF$ is not empty, there is $\scrK' \in \calF$ such that $F \in \scrK'$. As $\scrK'$ satisfies \eqref{K1}, $F$ cannot be among its elements, we have~a contradiction. Thus, $F \notin \scrK$.

\smallskip

\eqref{K2} Let $x \in B$ with $x \neq\zero$. We need to show that $\{x\} \in \scrK$. Consider $G_1, \dots, G_n \in 2^{B}_+$ such that $\{x\} \preceq G_1+ \ldots + G_n$. So, $x$ is a nonzero lower bound of $G_1+ \ldots + G_n$. If $\bigwedge G_i=\zero$ for every $i \leq n$, \citep[Lemma 1.33]{kop89} implies that $\bigwedge (G_1+ \ldots + G_n)=\zero$, which is not possible because $G_1+ \ldots + G_n$ has a nonzero lower bound. Therefore, one of the $G_i$ has a nonzero lower bound, and so is an element of $\Kmin$. But $\Kmin$ is a subset of every ultracontact on $B$, so---by the assumption---$\Kmin\subseteq\bigcap\calF$. In consequence, $G_i \in \bigcap \calF$. We conclude that $\{x\} \in \scrK$.

\smallskip

\eqref{K3} Let $F_1 \in \scrK$ and $F_2 \in 2^{B}_+$ be such that $F_1 \preceq F_2$. Consider $G_1, \dots ,G_n \in 2^{B}_+$ and suppose that $F_2 \preceq G_1+ \ldots + G_n$. Since $F_1 \preceq F_2$ and $\preceq$ is transitive, we obtain that $F_1 \preceq G_1+ \ldots + G_n$. From the first assumption it follows that $G_i \in \bigcap \calF$ for some $i \leq n$. Therefore, $F_2 \in \scrK$.

\smallskip

\eqref{K4} Let $F_1,F_2 \in 2^{B}$ and assume that $F_1+F_2 \in \scrK$. Then $F_1\neq \emptyset\neq F_2$, as otherwise $\emptyset \in \scrK$, which is not possible by \eqref{K0}. Suppose  towards contradiction that $F_1 \notin \scrK$ and $F_2 \notin \scrK$.  Then there exist $G_1, \dots, G_n,H_1, \dots, H_m \in 2^{B}_+$ such that
\[
F_1 \preceq G_1+ \ldots + G_n \text{ and } F_2 \preceq H_1+ \ldots + H_m,
\]
but none of $G_1, \ldots, G_n,H_1, \ldots, H_m$ is in $\bigcap\calF$.
It follows from Proposition~\ref{prop:support-very-basic}(3) that
\[
F_1 + F_2 \preceq G_1+ \ldots + G_n + H_1 + \ldots + H_m.
\]
As $F_1+F_2 \in \scrK$, the definition of $\scrK$ implies that one of $G_1, \dots, G_n,H_1, \dots, H_m$ is in $\bigcap \calF$, which is a~contradiction.
This means that $F_1 \in \scrK$ or $F_2 \in \scrK$. We conclude that $\scrK$ is an ultracontact on~$B$.

\smallskip

It remains to show that $\scrK$ is the largest ultracontact on $B$ such that $\scrK\subseteq\bigcap\calF$. Let then $\scrK'$ be an ultracontact such that $\scrK' \subseteq \bigcap \calF$. We prove that $\scrK' \subseteq \scrK$. Consider $F \in \scrK'$ and $G_1, \dots, G_n \in 2^{B}_+$ with $F \preceq G_1+ \ldots+ G_n$. As $\scrK'$ satisfies \eqref{K0}, we have $F \neq \emptyset$. Since $\scrK'$ satisfies \eqref{K3} and \eqref{K4-general}, we obtain that $G_i \in \scrK'$ for some $i \leq n$. As  $\scrK' \subseteq \bigcap \calF$, it follows that $G_i \in \bigcap \calF$. Then the definition of $\scrK$ yields $F \in \scrK$. Thus $\scrK' \subseteq \scrK$, which ends the whole proof.
\end{proof}

\pagebreak

For chains of ultracontacts, infima coincide with set-theoretical intersections.
\begin{theorem}
    If $\calC$ is a non-empty chain of ultracontacts, then $\bigcap\calC$ is an ultracontact, i.e., $\bigwedge\calC=\bigcap\calC$.
\end{theorem}
\begin{proof}
The non-trivial condition to verify is \eqref{K4}. To this end suppose that $F+G$ is in every ultracontact in $\calC$ but for $F$ there is a $\scrK_1\in\calC$ with $F\notin\scrK_1$ and for $G$ there is a $\scrK_2\in\calC$ with $G\notin\scrK_2$. Since $F+G$ is in $\scrK_1\cap\scrK_2$ by assumption, we have that $G\in\scrK_1$ and $F\in\scrK_2$.  However, $\calC$ is a chain, so either $\scrK_1\subseteq\scrK_2$, and thus $G\in\scrK_2$, or $\scrK_2\subseteq\scrK_1$ in which case $F\in\scrK_1$. Thus, in both cases we have a~contradiction.
\end{proof}
Further, in Theorem~\ref{th:meet-of-UC-is-not-intersection}, we will prove that in general the infimum of a family of ultracontacts is not the set-theoretical intersection. For now, we will make use of Theorem \ref{th:meet-ultracontact} to show that the lattice of all ultracontacts is a complete co-Heyting algebra (a co-frame, in the terminology of frames and locales). For this, we will need a lemma.
\begin{lemma}\label{lem:aux-for-co-Heyting} If $G,H_1,\ldots ,H_m$ are subsets of a Boolean algebra $B$, then
\[
    G\cup (H_1+\ldots+H_m)\preceq(G\cup H_1)+\ldots+(G\cup H_m).
\]
\end{lemma}
\begin{proof}
    Assume that
    \[
        x_1+\ldots+x_m\in (G\cup H_1)+\ldots+(G\cup H_m),
    \]
    where for each $i\leqslant m$, $x_i\in G\cup H_i$.
    If for all $i\leqslant m$, $x_i\in H_i$, then $x_1+\ldots+x_m\in H_1+\ldots+H_m$, whence, certainly
    \[
        x_1+\ldots+x_m\in G\cup (H_1+\ldots+H_m).
    \]
    Otherwise, there is an $i_0\leqslant m$ such that $x_{i_0}\notin H_{i_0}$, and so $x_{i_0}\in G$. But then
    \[
        x_{i_0}\leqslant x_1+\ldots + x_m
    \]
    and $x_{i_0}\in G\cup(H_1+\ldots+H_m)$. In either case, we have shown that $G\cup (H_1+\ldots+H_m)$ supports $(G\cup H_1)+\ldots+(G\cup H_m)$.
\end{proof}

\begin{theorem}\label{th:UC-co-frame}
    The lattice of ultracontacts satisfies infinite meet distributivity, i.e., for any non-empty family $\{\scrK_i\mid i\in I\}$
    \[\scrK\cup\bigwedge_{i\in I}\scrK_i=\bigwedge_{i\in I}(\scrK\cup\scrK_i).\]
\end{theorem}
\begin{proof}
Let $\{\scrK_i\mid i\in I\}\cup\{\scrK\}$ be a family of ultracontacts. The inclusion
\[
\scrK\cup\bigwedge_{i\in I} \scrK_i\subseteq \bigwedge_{i\in I} (\scrK\cup\scrK_i)
\]
is clear. For the reverse one, suppose that $G\in \bigwedge_{i\in I} (\scrK \cup \scrK_i)$. Thus, $G$ is non-empty and contraposing the condition from Theorem~\ref{th:meet-ultracontact} we obtain that
\begin{equation}\label{e: meet of Si cup T}\tag{$\dagger$}
    \text{$G\npreceq G_1+ \ldots+G_n$ for all non-empty $G_1,\ldots,G_n\notin \bigcap_{i\in I} (\scrK \cup \scrK_i)$.}
\end{equation}
Let us assume that $G\notin \scrK$. We will prove that $G\in \bigwedge_{i\in I} \scrK_{i}$, that is, we are going to show that
\begin{equation}\label{e: meet of Si}\tag{$\ddagger$}
    \text{$G\npreceq H_1+ \ldots+H_m $ for all non-empty $H_1,\ldots,H_m\notin \bigcap_{i\in I} \scrK_i$.}
\end{equation}
Suppose, then, that there is a family of non-empty sets $H_1,\ldots,H_m\notin \bigcap_{i\in I} \scrK_i$. Since $G\notin \scrK$, and since $G\cup H_i\preceq G$ for every $i\leqslant m$, none of the unions $G\cup H_1,\ldots,G\cup H_m$ is in $\scrK$, by \eqref{K3}. As each $\scrK_i$ is upward-closed with respect to $\preceq$, then so is their intersection $\bigcap_{i\in I} \scrK_i$. Once again, then, since $H_1,\ldots,H_m\notin \bigcap_{i\in I} \scrK_i$, and $G\cup H_i\preceq H_i$ for all $i\leqslant m$, also $G\cup H_1,\ldots,G\cup H_m\notin \bigcap_{i\in I} \scrK_i$. Therefore
\[
G\cup H_1,\ldots,G\cup H_m\notin \scrK\cup \bigcap_{i\in I}\scrK_i = \bigcap_{i\in I} (\scrK \cup \scrK_i).
\]
By \eqref{e: meet of Si cup T}, it must be the case that
\[
G \npreceq (G\cup H_1)+\ldots+ (G\cup H_m).
\]
From this, by Lemma~\ref{lem:aux-for-co-Heyting} and by the transitivity of $\preceq$ we obtain
 \[
 G \npreceq G\cup (H_1+\ldots+H_m).
 \]
 By the definition of $\preceq$, there is an $x\in G\cup (H_1+\ldots+H_m)$ such that $g\nleq x$ for all $g\in G$. As $x$ cannot be in $G$, then it is in $H_1+\ldots+H_m$, and so
\[
G \npreceq H_1+\ldots+H_m.
\]
This means that \eqref{e: meet of Si} holds, as desired.
\end{proof}

From Theorems \ref{th:lattice-of-UCAs} and \ref{th:UC-co-frame} we obtain a~result analogous to the one concerning the lattice of contact relations \cite[Corollary 1]{dw_cl}.

\begin{corollary}
    The lattice of all ultracontacts of any Boolean algebra $B$ is a complete co-Heyting algebra.
\end{corollary}

\section{New ultracontacts from old ones}\label{sec:new-UCs}

In this section, we will focus on the problem of how we can expand existing ultracontacts to obtain <<bigger>> ones.

 If $\scrK$ is an ultracontact on $B$ and $M\subseteq B$, let
 \[
  \scrK^{M}\defeq\scrK\cup\powerne{M}.
\]
If $\calM$ is a collection of subsets of $B$, we put
\[
    \scrK^{\calM}\defeq\scrK\cup \bigcup\{\scrK^{M}\mid M\in\calM\}.
\]

\begin{lemma}\label{lem:K^M-the-smallest}
    Let $\klam{B,\scrK}$ be a UCA, and let $M$ be a subset of $B$.
    \begin{enumerate}[label=(\arabic*),ref=\arabic*,itemsep=0pt]
        \item If $M\in\scrK$, then $\scrK^{M}=\scrK$.
        \item If $M$ has a non-zero lower bound, then $\scrK^{M}=\scrK$.
        \item If $\scrK^M$ is an ultracontact and $M\neq\emptyset$, then it is the smallest ultracontact extending $\scrK$ and containing~$M$.
    \end{enumerate}
\end{lemma}
\begin{proof}
(1) If $F\subseteq M$ and $F\neq \emptyset$, then $F\in \scrK$, by Lemma \ref{lem:K-down-closed}. Then, $2^M_+\subseteq \scrK$.

\smallskip

(2) If $M$ has a nonzero lower bound, so does every nonempty subset of $M$. In consequence, $2^M_+\subseteq \Kmin\subseteq \scrK$.
So, the definition of $\scrK^{M}$ yields that $\scrK^{M} = \scrK$.

\smallskip

(3) Suppose $\scrK\subseteq\scrK'$ and $M\in\scrK'$. By Lemma~\ref{lem:K-down-closed} every non-empty subset of $M$ is an element of $\scrK'$. In consequence $\scrK^M\subseteq\scrK'$.
\end{proof}

\begin{theorem}\label{th:UC-from-grills}
    For any ultracontact algebra $\klam{B,\scrK}$ and any collection of grills~$\calG$, $\scrK^{\calG}$ is an ultracontact. Moreover, $\scrK^{\calG}$ is the smallest ultracontact extending $\scrK$ and containing all the grills from $\calG$.
\end{theorem}
\begin{proof}
By Theorem~\ref{th:lattice-of-UCAs}, it is enough to show that for any $G\in\calG$, $\scrK^G$ is an ultracontact. The axioms \eqref{K0}--\eqref{K2} are obvious.

\smallskip

For \eqref{K3} suppose that $F\preceq M\neq\emptyset$ and $F \in \scrK^G$. If $F \in \scrK$, then $M\in \scrK \subseteq \scrK^G$ because $\scrK$ is an ultracontact. Consider the case when $F \subseteq G$. By Proposition~\ref{prop:support-very-basic}(6), $M\subseteq\upop F$, and $\upop F\subseteq G$, since $G$ is a~grill. Thus $M\subseteq G$ and so $M\in\scrK^G$.

\smallskip

For \eqref{K4}, assume that $F+H\in\scrK^G$. If $F+H\in\scrK$, we are done, as $\scrK$ is an ultracontact which is a subset of $\scrK^G$. If $F+H\subseteq G$, then by Proposition~\ref{prop:prime-for-big-union}, $F\subseteq G$ or $H\subseteq G$. Since $F+H\neq\emptyset$, both $F$ and $H$ are non-empty, so $F\in\scrK^G$ or $H\in\scrK^G$, as required.

\smallskip

For the proof of the second part of the theorem, assume that (a) $\scrK\subseteq\scrK'$ and let (b)~$\calG\subseteq\scrK'$. If $M\in\scrK^{\calG}$, then either $M\in\scrK$ or there is a~grill $G\in\calG$ such that $\emptyset\neq M\subseteq G$. In the first case, $M\in\scrK'$ by~(a). In the second case, $G\preceq M$, so $M\in\scrK'$, as $G$ is an element of $\scrK'$ by~(b). So, in both cases $M \in \scrK'$, and hence $\scrK^{\calG}\subseteq\scrK'$, as desired.
\end{proof}

\begin{lemma}\label{lem:intersection-for-stacks}
For any $U,V\in\Stack(B)$, $U\cap V=U+V$.
\end{lemma}
\begin{proof}
($\subseteq$) If $x\in U\cap V$, then $x=x+x\in U+V$.

\smallskip

($\supseteq$) Let $w\in U+V$, which means that for some $x\in U$ and $y\in V$, $w=x+y$. As $x,y\leq w$ we have that $w\in\upop U$ and $w\in\upop V$.
By the assumption, both $U$ and $V$ are stacks, so $w\in U\cap V$, as required.
\end{proof}

\begin{theorem}\label{th:M-not-grill-not-closed-for-infima}
     Let $\klam{B,\scrK}$ be an ultracontact algebra and $\emptyset\neq M\in\Stack(B)\setminus\{B\}$ be such that $M\notin\scrK$. If $M$ is not a~grill, then there are ultracontacts $\scrK_1$ and $\scrK_2$ extending $\scrK$ and such that $M\in\scrK_1\cap\scrK_2$ but $M\notin\scrK_1\wedge\scrK_2$. In consequence,
    the collection of all ultracontacts that contain $M$ is not closed under infima.
\end{theorem}
\begin{proof}
    Let $M$ be a non-empty stack of $B$ such that $M\notin\scrK$. Let $x$ and $y$ be such that (a) $x+y\in M$, but $x\notin M$ and $y\notin M$. Let $M_x\defeq M\cup\upop x$ and $M_y\defeq M\cup\upop y$. As $M\notin\scrK$, and $M_x,M_y\preceq M$, neither $M_x$ nor $M_y$ is in $\scrK$.

    Observe that since $M$ is a stack, $M=M_x\cap M_y$. Indeed, if $a\in\upop x\cap\upop y$, then $x+y\leq a$, and so $a\in M$, by~(a). By Lemma~\ref{lem:intersection-for-stacks} we obtain that (b) $M=M_x+M_y$.

    Further, $G_1\defeq B\setminus\downop x$ and $G_2\defeq B\setminus\downop y$ are grills, as $x\neq\one\neq y$, and so $\downop x$ and $\downop y$ are proper ideals. In consequence, by Theorem~\ref{th:UC-from-grills}, the families $\scrK^{G_1}$ and $\scrK^{G_2}$ are ultracontacts. Moreover, as $M$ is a stack such that $x\notin M$, $M\cap\downop x=\emptyset$. So $M\subseteq G_1$, and therefore $M\in\scrK^{G_1}$. By a~similar argument, $M\subseteq G_2$, and thus $M\in \scrK^{G_1}\cap\scrK^{G_2}$.

    Note that $M_x$ cannot be an element of $\scrK^{G_1}$, as it is neither an element of $\scrK$ nor a subset of $G_1$ (observe that $x\in M_x\setminus G_1$). Similarly, $M_y$ cannot be an element of $\scrK^{G_2}$. So, neither $M_x$ nor $M_y$ is in $\scrK^{G_1}\cap\scrK^{G_2}$. Since, in consequence of (b), $M\preceq M_x+M_y$, Theorem~\ref{th:meet-ultracontact} entails that $M\notin\scrK^{G_1}\wedge\scrK^{G_2}$.
\end{proof}

\begin{lemma}\label{lem:UC-closed-for-similar}
    For any Boolean algebra $B$ and any ultracontact $\scrK$ on $B$, $\scrK$ is closed under the similarity relation, that is, if $F\sim G$, then: $F\in\scrK$ iff $G\in\scrK$.
\end{lemma}
\begin{proof}
    By \eqref{K0} and \eqref{K3}.
\end{proof}

\begin{lemma}\label{lem:proper-ext->stack}
    Let $\klam{B,\scrK}$ be a UCA and $\emptyset\neq M\notin\scrK$. If $\scrK^M$ is an ultracontact, then $M$ is a~stack. Moreover, if $\emptyset\neq F+G\subseteq M$ but $F+G\notin\scrK$, then $F\subseteq M$ or $G\subseteq M$.
\end{lemma}
\begin{proof}
    Suppose $M\neq\emptyset$ and  $M\notin\scrK$. From Lemma \ref{lem:UC-closed-for-similar} and Proposition \ref{prop:support-very-basic}(9) we obtain that $\upop M\notin\scrK$. But $M$ is an element of $\scrK^M$ by the definition, and $\scrK^M$ is an ultracontact by the assumption. So, by Lemma~\ref{lem:UC-closed-for-similar}, we have that $\upop M\in\scrK^M$. Therefore $\upop M\subseteq M$, and so $M$ is a~stack.

    For the second part, let $\emptyset\neq F+G$ be a subset of $M$ which is not an element of $\scrK$. Thus, by Proposition~\ref{prop:K3-reversed} we obtain that neither $F$ nor $G$ is an element of $\scrK$. But $F+G\in\scrK^M$, so either $F\in\scrK^M$ or $G\in\scrK^M$, by \eqref{K4}. In consequence, $F\subseteq M$ or $G\subseteq M$.
\end{proof}

\begin{theorem}\label{t: K^M ultracontact iff M grill}
Let $\klam{B,\scrK}$ be an ultracontact algebra and $\emptyset\neq M\notin\scrK$. Then, $\scrK^M$ is an ultracontact iff $M$ is a grill.
\end{theorem}
\begin{proof}
    The right-to-left implication is a~special case of Theorem~\ref{th:UC-from-grills}. For the other direction, assume that $\scrK^M$ is an ultracontact but $M$ is not a~grill. Since $\scrK^M$ is a~proper extension of $\scrK$, from Lemma~\ref{lem:proper-ext->stack} we obtain that $M$ is a~stack. $M\neq B$, since $M\in\scrK^{M}$ and $\scrK^{M}$ satisfies \eqref{K1}. As $M\notin\scrK$ and $M\neq\emptyset$, by  Theorem~\ref{th:M-not-grill-not-closed-for-infima} there are ultracontacts $\scrK_1$ and $\scrK_2$ extending $\scrK$ such that both contain $M$, but $M\notin\scrK_1\wedge\scrK_2$. It follows that no ultracontact $\scrK'$ such that $\scrK'\subseteq\scrK_1\cap\scrK_2$ contains~$M$. However, according to Lemma~\ref{lem:K^M-the-smallest}, $\scrK^M$ is the smallest ultracontact that contains $M$, so $\scrK^M\subseteq \scrK_1\cap\scrK_2$, which is a contradiction.
    \end{proof}

Adding the members of $\powerne{G}$, where $G$ is a grill, to an ultracontact $\scrK$ is not the only way to obtain UCAs that properly extend $\scrK$. Trivially, $\powerne{G}\setminus\{\{m\}\mid m\in G\}$ gives the same extension as $\powerne{G}$. But there are more interesting cases.
\begin{example}
    Consider the algebra $\FC(\omega)$ of all finite and co-finite subsets of the set of natural numbers and its only free ultrafilter $U$, the Fr\'{e}chet filter. Let $\scrZ\defeq\{Z\subseteq U\mid\bigwedge Z=\zero\}$. Clearly, only infinite subsets of $U$ are in $\scrZ$, but not all such sets (for example the family of co-finite sets that contain all even numbers is not in $Z$). So, $\Kmin\cup\scrZ=\Kmin\cup\powerne{U}$ and thus $\Kmin\cup\scrZ$ is an ultracontact, but $\scrZ$ is not of the form $\powerne{M}$ for any set~$M$.

    In general, if $B$ is an infinite Boolean algebra, then it has a~non-principal ultrafilter~$U$. $\scrK\defeq\Kmin\cup\powerne{U}$ is then the same as
    \[
        \Kmin\cup\left\{Z\subseteq U\suchthat \bigwedge Z=\zero\right\}.
    \]
    The set $\{Z\subseteq U\mid \bigwedge Z=\zero\}$ contains only infinite subsets of $U$. Again, the right-hand component of the union above is never of the form $\powerne{M}$, for any $M$.

    Generalizing a bit further, for any grill $G$, $\scrK_{\min} \cup \{Z \in 2^G_+ \mid \bigwedge Z= \zero\}$ is the same as $\Kmin\cup\powerne{G}$. \QED
\end{example}

\begin{proposition}\label{prop:UC-from-grill-principal}
    If $\klam{B,\scrK}$ is a UCA and $\ca{P}$ is a family of principal ultrafilters of $B$, then
    \[
    \scrK^{\ca{P}}=\scrK.
    \]
    Moreover, if $\ca{X}\subseteq\Grill(B)$ and $\ca{P}$ is a subfamily of $\ca{X}$ whose elements are principal ultrafilters, then
    \[
    \scrK^{\ca{X}}=\scrK^{\ca{X}\setminus\ca{P}}.
    \]
\end{proposition}
\begin{proof}
    It is enough to observe that each member of $\ca{P}$ has a nonzero lower bound because it is a principal ultrafilter. So, Lemma~\ref{lem:K^M-the-smallest}(2) implies that $\scrK^U=\scrK$ for all $U \in \ca{P}$. We conclude that $\scrK^\ca{P}=\scrK$ and $\scrK^\ca{X} = \scrK^{\ca{X} \setminus \ca{P}}$.
\end{proof}

\begin{lemma}\label{lem:Kmin-from-atoms}
    If $\klam{B,\scrK}$ is a UCA and $A\subseteq B$ is a non-empty collection of atoms, then the smallest ultracontact extending $\scrK$ and containing $A$ is $\scrK\cup \Upop A$.
    In particular, the smallest ultracontact containing $A$ is $\Kmin\cup\Upop A$.
\end{lemma}
    \begin{proof}
    Notice that $\Upop A=\powerne{\bigcup\{\upop a\mid a\in A\}}$ and as each $a\in A$ is an atom, $\bigcup\{\upop a\mid a\in A\}$ is a~grill. Thus, Theorem~\ref{th:UC-from-grills} implies that $\scrK\cup \Upop A$ is an ultracontact.

     For minimality, observe that if $\scrK'$ is a UCA extending $\scrK$ such that $A\in\scrK'$, then $\Upop A\subseteq\scrK'$, by \eqref{K3}. So, $\scrK\cup \Upop A\subseteq\scrK'$.
     \end{proof}

Note that the method of extending $\scrK$ from Lemma~\ref{lem:Kmin-from-atoms}
 may fail if $A$ is not a~set of atoms. Below is an example showing this.

\begin{example}\label{ex:non-atoms-failure}
Consider a Boolean algebra $B$ and a set $\{a,b\}\subseteq B$ such that both its elements are non-zero, $a$~is not an atom, and $a\cdot b=\zero$. So there are $a_0$ and $a_1$, both non-zero, such that $a_0\cdot a_1=\zero$ and $a_0+a_1=a$. We have that
    \[
\{a_0,b\}+\{a_1,b\}=\{a_0+a_1,a_0+b,b+a_1,b\}=\{a,a_0+b,b+a_1,b\}.
    \]
    As $\{a,b\}\preceq \{a,a_0+b,b+a_1,b\}$, it follows that
    \[
        \{a,a_0+b,b+a_1,b\}\in\Kmin\cup\Upop\{a,b\}.
    \]
    Yet neither $\{a_0,b\}$ nor $\{a_1,b\}$ can be in $\Kmin\cup\Upop\{a,b\}$, so the family fails to satisfy \eqref{K4} and so is not a~UC.\QED
\end{example}

\begin{theorem}\label{th:meet-of-UC-is-not-intersection}
    In general, the set of UCs on a Boolean algebra $B$ is not closed under intersections.
\end{theorem}
\begin{proof}[Proof by example]
Consider the sixteen-element algebra $B$ whose atoms are $a$, $b$, $c$, and $d$. By Lemma~\ref{lem:Kmin-from-atoms}, the following two families on $B$
    \[
    \scrK_{\{a,b\}}\defeq\Kmin\cup\Upop\{a,b\}\qtand \scrK_{\{c,d\}}\defeq\Kmin\cup\Upop\{c,d\}
    \]
    are the smallest UCAs containing $\{a,b\}$ and $\{c,d\}$ respectively.
    Observe that
    \[
    \{a,b\}\notin\scrK_{\{c,d\}}\qtand \{c,d\}\notin\scrK_{\{a,b\}}.
    \]
We have that
    \[
\{a,b\}+\{c,d\}=\{a+c,a+d,b+c,b+d\}.
  \]
Since by Proposition~\ref{prop:support-very-basic}(1) both $\{a,b\}$ and $\{c,d\}$ support $\{a,b\}+\{c,d\}$ it must be the case that
\[
\{a,b\}+\{c,d\}\in\scrK_{\{a,b\}}\cap\scrK_{\{c,d\}},
\]
but by the observation above, none of the two summands is in the intersection. It follows that \eqref{K4} does not hold for $\scrK_{\{a,b\}}\cap\scrK_{\{c,d\}}$, which in consequence is not a~UC.
\end{proof}

\pagebreak

\section{Contact relations from UCs}\label{sec:C-from-UC}

A \emph{contact relation} $\con$ on a Boolean algebra $\klam{B,+, \cdot, -, \zero,\one}$ is a binary relation $\mathord{\con}\subseteq B\times B$ that satisfies the following axioms ($\notcon$ is the set-theoretical complement of $\con$):
\begin{gather}
    (\forall x\in B)\,\zero\notcon x,\tag{C0}\label{C0}\\
    (\forall x\in B)\,(x\neq\zero\rightarrow x\con x),\tag{C1}\label{C1}\\
    (\forall x,y\in B)\,(x\con y\rightarrow y\con x),\tag{C2}\label{C2}\\
    (\forall x,y,z\in B)\,(x\con y\wedge y\leq z\rightarrow x\con z),\tag{C3}\label{C3}\\
    (\forall x,y,z\in B)\,(x\con y+z\rightarrow x\con y\vee x\con z).\tag{C4}\label{C4}
\end{gather}
A \emph{Boolean contact algebra} is a pair $\klam{B,\con}$ such that $B$ is a Boolean algebra, and $\con$ is a~contact relation. In the case $x\con y$ we say that $x$ \emph{is in contact with} $y$. Since the axioms \eqref{C0}--\eqref{C4} are universal, each restriction of $\con$ to a subalgebra of $B$ is also a contact relation.

Each UCA $\klam{B,\scrK}$ gives rise to the contact relation determined by $\scrK$ in a natural way
\begin{equation}\tag{$\dftt{\Kon}$}\label{df:C_K}
    x\Kon y\iffdef \{x,y\}\in\scrK.
\end{equation}

\begin{theorem}\label{thm:UCtoC}
    If $B$ is a Boolean algebra and $\scrK$ is an ultracontact, then $\Kon$ is a contact relation.
\end{theorem}
\begin{proof}
We check all the axioms for contact.

\eqref{C0} This holds as any pair of the form $\{\zero,x\}$  cannot be in $\scrK$ by \eqref{K1}.

\smallskip

\eqref{C1} By \eqref{K2}, if $x\neq\zero$, then $\{x\}\in \scrK$, so $x\Kon x$.

\smallskip

\eqref{C2} Immediate.

\smallskip

\eqref{C3} Suppose that $x\Kon y$ and $y\leq z$. Clearly, $\{x,y\}\preceq \{x,z\}$, and so $\{x,z\}\in\scrK$, as desired.

\smallskip

\eqref{C4} Suppose that $\{x,x_1+ x_2\}\in \scrK$. We have
\[
\{x,x_1+ x_2\}\preceq \{x,x+ x_1,x+ x_2,x_1+ x_2\}.
\]
Then, by \eqref{K3}, $\{x,x+ x_1,x+ x_2,x_1+ x_2\}\in \scrK$. But
\[
\{x,x+ x_1,x+ x_2,x_1+ x_2\}=\{x,x_1\}+ \{x,x_2\},
\]
so, by \eqref{K4}, either $\{x,x_1\}\in \scrK$ or $\{x,x_2\}\in \scrK$.
\end{proof}

Theorem~\ref{th:smallest-UC} can be interpreted as the counterpart of the fact that the overlap relation
\[
x\Overl y\iffdef x\cdot y\neq\zero
\]
is the smallest contact relation, i.e., the smallest relation that satisfies axioms \eqref{C0}--\eqref{C4}. Theorem~\ref{th:largest-UC} is, on the other hand, the counterpart of the fact that the full relation on $\Bne$, $B^{+}\times B^{+}$, is the largest contact. Moreover, the contact relations induced by $\Kmin$ and $\Kmax$ in the sense of \eqref{df:C_K} are, respectively, the overlap relation and the full relation $\Bne\times\Bne$ of the algebra.

\begin{theorem}\label{th:C_S}
For the smallest UC on B we have
\[
x\con_{\Kmin}y\iffslim x\cdot y\neq\zero,
\]
and for the largest one, it is the case that
\[
x\con_{\Kmax}y\iffslim x\neq\zero\neq y.
\]
\end{theorem}
\begin{proof}
    For $\Kmin$ we have: $x\con_{\Kmin}y$ if and only if $\{x,y\}$ has a non-zero lower bound if and only if the meet of $x$ and $y$ is different from $\zero$.

    For $\Kmax$:  $x\con_{\Kmax}y$ if and only if $\{x,y\}\in\Kmax$ if and only if $\zero\notin\{x,y\}$.
\end{proof}

\begin{example}\label{ex:diff-UCs-same-C}
    Let us show that there are different ultracontacts that induce the same contact relation. To this end, we proceed with the 8-element algebra $B$ with atoms $\{a,b,c\}$. We start with an observation that the minimal ultracontact of $B$ is
\[
\Kmin=\powerne{\upop a}\cup\powerne{\upop b}\cup \powerne{\upop c}.
\]
The algebra has three ultrafilters: $\upop a$, $\upop b$, and $\upop c$. Since any union of ultrafilters is a grill, it follows that
\[
\Grill(B)=\{\upop a,\upop b,\upop c,\upop a\cup\upop b,\upop a\cup\upop c,\upop b\cup \upop c,\upop a\cup\upop b\cup\upop c\}.
\]
We have $2^7-1$ non-empty subsets of the set of grills, but by Proposition~\ref{prop:UC-from-grill-principal} not all of them give rise to ultracontacts that differ from $\Kmin$. By the same corollary, the two sets below (and others that contain ultrafilters of $B$)
\[
\{\upop a\cup \upop b\}\qtand \{\upop a,\upop a\cup \upop b\}
\]
will result in the same $\scrK$. It is easy to see that the following are the minimal UCs on $B$ extending $\Kmin$:
\begin{align*}
\scrK_{a,b}\defeq{}&\Kmin\cup \powerne{\upop a\cup\upop b},\\
\scrK_{a,c}\defeq{}&\Kmin\cup \powerne{\upop a\cup\upop c},\\
\scrK_{b,c}\defeq{}&\Kmin\cup \powerne{\upop b\cup\upop c}.
\end{align*}
In the case of $\scrK_{a,b}$, since $\{a,b\}\subseteq\upop a\cup\upop b$, we have that $a\con_{\scrK_{a,b}}b$. The other non-overlapping pairs of objects of $\con_{\scrK_{a,b}}$ are
\[
\klam{a,b+c}\qtand\klam{a+c,b}.
\]
Similar observations can be made about $\scrK_{a,c}$ and $\scrK_{b,c}$ and their corresponding contact relations. We can see that these are three minimal UCs on $B$ extending~$\Kmin$.

Different UCs on $B$ that generate the same contact relation are
\begin{align*}
    \scrK\defeq{}&\Kmin\cup \powerne{\upop a\cup\upop b}\cup\powerne{\upop a\cup\upop c}\cup\powerne{\upop b\cup\upop c}\qtand\\
    \scrK_{a,b,c}\defeq{}&\Kmin\cup\powerne{\upop a\cup\upop b\cup\upop c}.
\end{align*}
Both generate the full contact relation on $B$, however $\{a,b,c\}\in\scrK_{a,b,c}\setminus\scrK$. The difference is not purely formal, and we can meaningfully interpret it: In $\scrK$ all  pairs of non-zero regions are in contact (figuratively speaking: <<share a point>>); in $\scrK_{a,b,c}$ all triples of regions <<share a point>>, in particular there is a~<<point>> of contact of $a$, $b$, and $c$ (which is not the case for $\scrK$). Let us observe that $\scrK$ and $\scrK_{a,b,c}$ correspond to the configurations of regions in Figure~\ref{fig:three regions}.\QED
\end{example}

Let us encapsulate the above in the following theorem.

\begin{theorem}\label{th:UCs-with-the-same-contact}
There exists a Boolean algebra $B$ and different UCs $\scrK_1$ and $\scrK_2$ on $B$ for which $\con_{\scrK_1}=\con_{\scrK_2}$.
\end{theorem}

\begin{example}
    It is a consequence of the construction from Example~\ref{ex:diff-UCs-same-C}, that for $x\notin\{\zero,\one\}$ there may exist incomparable ultracontacts containing $\{x,-x\}$. For the same algebra as in the example and for  $x\defeq a$ and $-x=b+c$, define
    \[
        \scrK_1\defeq\Kmin\cup\powerne{\upop a\cup\upop b}\qtand \scrK_2\defeq\Kmin\cup\powerne{\upop a\cup\upop c}.
    \]
    Clearly, $\{a,b+c\}\in\scrK_1\cap\scrK_2$, but $\{a,b\}\in\scrK_1\setminus\scrK_2$ and $\{a,c\}\in\scrK_2\setminus\scrK_1$.\QED
\end{example}

Let us make one more observation that emphasizes the specific nature of the approach to nearness of regions presented in this paper. Given a BCA $\klam{B,\con}$ by a \emph{clique} we understand a non-empty set $M\subseteq B$ such that $M\times M\subseteq\con$. Let $\scrM$ be the family of all such cliques. Then, $\scrM$ satisfies \eqref{K0} by definition, \eqref{K1} by \eqref{C0}, \eqref{K2} by \eqref{C1}, and \eqref{K3} by \eqref{C3}. So, up to \eqref{K3}, the axioms for UCAs work well when the elements of $\scrK$ are cliques. However, the connection breaks at \eqref{K4}, which is shown by the example that follows. Going back to our motivations for ultracontact, this may be seen as an abstraction from the fact that a family of subsets $\ca{F}\subseteq\ca{P}(X)$ may be such that $F_1\cap F_2\neq \emptyset$ for each $F_1,F_2\in \ca{F}$ but $\bigcap \ca{F}=\emptyset$.

\begin{example}\label{ex:abcd} Fix a BCA $\klam{B,\con}$ and let $a,b,c,d$ be any four pairwise different elements of $B$ such that $b\neq\zero$ and
\[
a\con c\qtand a\con d
\]
but $a\notcon b$ and $c\notcon d$. Observe that
\[
\{a,b\}+\{c,d\} = \{a+c,a+d,b+c,b+d\}
\]
is a clique. Indeed, since $a\con c$, $a\leq a+d$ and $c\leq b+c$, by \eqref{C3} and the symmetry of the contact we have that $a+d\con b+c$. In the analogous way, we obtain from $a\con d$ that $a+c\con b+d$. All the remaining cases concern pairs that are in the overlap relation, and hence in $\con$. It follows that $\{a,b\}+\{c,d\}$ is a clique, but neither $\{a,b\}$ nor $\{c,d\}$ is. This shows that \eqref{K4} fails for cliques.\QED
\end{example}

\subsection{Hypercontact in the framework of UCA}

The motivation somewhat similar to ours is a driving factor behind the work of Lipparini \cite{Lipparini-HS}, where the author generalizes contact relations to finite collections of elements sharing <<a point>> of contact. Lipparini introduces the notion of a hypercontact lattice, but we limit ourselves to considering his concept in the setting of Boolean algebras.
\begin{definition}
    A pair $\klam{B,\Delta}$ is an \emph{additive hypercontact} algebra if $B$ is a Boolean algebra, and $\Delta$ is a collection of finite subsets of $B$ such that
    \begin{enumerate}[label=(H\arabic*),ref=H\arabic*,itemsep=0pt]
        \item\label{it:HC1} If $\zero\in F$, then $F\notin\Delta$.
        \item\label{it:HC2} If $x\neq\zero$, then $\{x\}\in\Delta$.
        \item\label{it:HC3} If $F\in\Delta$ and $G\subseteq F$, then $G\in\Delta$.
        \item\label{it:HC4} If $F\in\Delta$ and there is an $x\in F$ such that $x\leq y$, then $F\cup\{y\}\in\Delta$.
        \item\label{it:HC5} If $F\cup\{x+y\}\in\Delta$, then $F\cup\{x\}\in\Delta$ or $F\cup\{y\}\in\Delta$.
    \end{enumerate}\QED
\end{definition}

\begin{proposition}
    Let $\klam{B,\scrK}$ be an ultracontact algebra. For
    \[
  \Delta_{\scrK}\defeq\{F\in \scrK\mid F\text{ is finite}\}\cup \{\emptyset\},
    \]
    the pair $\klam{B,\Delta_{\scrK}}$ is an additive hypercontact algebra.
\end{proposition}
\begin{proof}
  Axioms \eqref{it:HC1} and \eqref{it:HC2} are obvious.

  \smallskip

  For \eqref{it:HC3}, suppose that $F\in \Delta_{\scrK}$, and that $G\subseteq F$. If $F$ is empty, then $G$ is empty, and so $G\in \Delta_{\scrK}$. If $F\neq \emptyset$, then for any non-empty $G\subseteq F$, we have that $F\preceq G$ and so $G\in\scrK$ by \eqref{K3}. Since $G$ is finite, $G\in\Delta_\scrK$.

  \smallskip

  For \eqref{it:HC4}, suppose that $F\in\Delta_{\scrK}$ and there is an $x\in F$ such $x\leq y$. Then, $F\preceq F\cup\{y\}$. Thus, $F\cup\{y\}\in\scrK$ by \eqref{K3}, and as $F\cup\{y\}$ is finite $F\cup\{y\}\in\Delta_{\scrK}$.

  \smallskip

  For \eqref{it:HC5}, assume that $F\cup\{x+y\}\in\Delta_{\scrK}$. It is the case that $F\cup\{x+y\}\preceq (F\cup\{x\})+(F\cup\{y\})$. Indeed, every element of $(F\cup\{x\})+(F\cup\{y\})$ is either of the form $x+y$, or $a+x$, or $a+y$ for some $a\in F$. Then, by \eqref{K3}, $(F\cup\{x\})+(F\cup\{y\})\in\scrK$. By \eqref{K4}, $F\cup\{x\}\in\scrK$ or $F\cup\{y\}\in\scrK$. As both sets are finite, one of them is in $\Delta_{\scrK}$.
\end{proof}

To conclude, the results from this section imply that the approach put forward by us generalizes (\emph{mutatis mutandis}) various approaches to the nearness of regions known in the literature\footnote{\label{foot:ECA}There are two more approaches to region-based nearness known as extended contact algebras (see \cite{Ivanova-ECAATC}) and sequent algebras (see \cite{Vakarelov-AMBOSA}). Here, we do not examine their relation to UCAs, and we leave the comparison to future work.}.

\section{Finite UCAs and abstract simplicial complexes}\label{sec:fin-UCs-and-SCs}

We call an ultracontact algebra $\klam{B,\scrK}$ \emph{finite} when $B$ is finite. Below, we study the properties of such algebras, thanks to which we may put our work in a larger perspective by revealing a~connection between ultracontacts and abstract simplicial complexes.

\begin{lemma}\label{lem: finite ultracontact and atoms}
Let $\klam{B,\scrK}$ be a finite ultracontact algebra and $F \subseteq B$ be non-empty. Then $F \in \scrK$ if and only if there exists $G \subseteq \Atom(B)$ such that $G \preceq F$ and $G \in \scrK$.
\end{lemma}
\begin{proof}
The implication from right to left is an immediate consequence of \eqref{K3}. To establish the other one, consider $\{ a_1, \dots, a_n \} \in \scrK$. We need to prove that there exist $b_1, \dots, b_n \in \Atom(B)$ such that $\{ b_1, \dots, b_n \} \in \scrK$ and $b_i \leq a_i$ for every $i \leqslant n$. To this end, it is sufficient to show that for every $i \leqslant n$ there is a $b_i \in \Atom(B)$ such that $b_i \leq a_i$ and $\{ a_1, \dots, a_{i-1}, b_i, a_{i+1}, \dots, a_n \} \in \scrK$. Without loss of generality, we prove it for $i=1$. Since $B$ is finite, there exist $ c_1, \dots, c_m \in \Atom(B)$ for which $a_1=c_1 + \ldots + c_m$. Observe that
\begin{equation}\label{eq: support sum atoms}
\{a_1, \dots, a_n\} \preceq \{c_1, a_2, \dots, a_n\} + \ldots + \{c_m, a_2, \dots, a_n\}.
\end{equation}
Indeed, every element $d$ in the set of joins above either is the join of a collection that includes $a_i$ for some $2\leqslant i \leqslant n$, or not, in which case $d=c_1+\ldots+c_m$, and then $a_1\leqslant d$.
This establishes \eqref{eq: support sum atoms}. As $\scrK$ is an ultracontact, \eqref{eq: support sum atoms} together with \eqref{K3} and \eqref{K4-general} yield that there exists $c_j$ for some $j \leqslant m$ such that $\{c_j, a_2, \dots, a_n\} \in \scrK$. It is then sufficient to pick $b_1\defeq c_j$. This concludes the proof.
\end{proof}

The lemma lets us point to a connection of the central concept of this paper to the concept of \emph{simplicial complex} known from algebraic topology. Let us recall the definition.

\begin{definition}[{\citealp[p.\,108]{Spanier-AT}}]
    An \emph{abstract simplicial complex} is a pair $\klam{V,\Sigma}$ such that $V$ is a set whose elements are called \emph{vertices} and $\Sigma$ is a family of finite non-empty subsets of $V$, the elements of $\Sigma$ are called \emph{simplices} and satisfy the following two conditions:
    \begin{enumerate}[label=(SC\arabic*),ref=SC\arabic*]
        \item\label{it:SC1} every singleton subset of $V$ is in $\Sigma$ (for every $v\in V$, $\{v\}$ is a simplex),
        \item\label{it:SC2} If $S\in\Sigma$ and $\emptyset\neq T\subseteq S$, then $T\in\Sigma$ ($\Sigma$ is closed under non-empty subsets, i.e., every non-empty subset of a~simplex is a~simplex).
    \end{enumerate}
    Given a set of vertices $V$, by $\SC(V)$ we mean the collection of all simplicial complexes on $V$, i.e., all those $\Sigma\subseteq V$ for which $\klam{V,\Sigma}$ is an abstract simplicial complex.\QED
\end{definition}

\begin{lemma}\label{lem:UC-to-SC}
    If $B$ is a finite algebra, then for every $\scrK\in\UC(B)$, $\scrK\cap\powerne{\Atom(B)}\in\SC(\Atom(B))$.
\end{lemma}
\begin{proof}
\eqref{it:SC1} Let $a \in \Atom(B)$. Then $a \neq\zero$, and so $\{ a \} \in \scrK$ by \eqref{K2}. Therefore, $\scrK \cap 2^{\Atom(B)}_+$ contains all the singletons.

\smallskip

\eqref{it:SC2} Consider $F,G \in 2^{\Atom(B)}_+$ such that $F \subseteq G$ and $G \in \scrK$. Then $G \preceq F$ by Proposition~\ref{prop:support-very-basic}(7), and by \eqref{K3} we obtain that $F \in \scrK$. So, $\scrK \cap 2^{\Atom(B)}_+$ is closed for non-empty subsets.
\end{proof}

\begin{theorem}\label{thm:correspondence ultracontact finite}
If $B$ is a finite algebra, then the mapping
\begin{equation*}
\sigma\colon\UC(B)\to\SC(\Atom(B))\quad\text{such that}\quad \sigma(\scrK)\defeq\scrK\cap\powerne{\Atom(B)}
\end{equation*}
is an order isomorphism (w.r.t. the set-theoretical inclusion).
\end{theorem}
\begin{proof}
Thanks to Lemma~\ref{lem:UC-to-SC}, $\sigma$ is well-defined, and it is obvious that it preserves inclusions. In order to show that it also reflects inclusions, suppose that $\scrK$ and $\scrK'$ are two ultracontacts on $B$ such that $\scrK \cap 2^{\Atom(B)}_+ \subseteq \scrK' \cap 2^{\Atom(B)}_+$. Consider $F \in \scrK$. Then Lemma~\ref{lem: finite ultracontact and atoms} yields $G \subseteq \Atom(B)$ such that  $G \preceq F$ and $G \in \scrK$. By \eqref{K0}, $F \neq \emptyset$, and hence $G \preceq F$ implies that $G \neq \emptyset$ as well. So, $G \in \scrK \cap 2^{\Atom(B)}_+$. From $\scrK \cap 2^{\Atom(B)}_+ \subseteq \scrK' \cap 2^{\Atom(B)}_+$ it follows that $G \in \scrK'$. Since $\scrK'$ is an ultracontact and $G \preceq F$, we obtain that $F \in \scrK'$ by \eqref{K3}. This shows that $\scrK \subseteq \scrK'$.

Since $\sigma$ reflects inclusions, it is injective. It remains to show that it is also surjective. Let $\klam{\Atom(B),\Sigma}$ be an abstract simplicial complex. Define $\scrK \subseteq 2^B_+$ by setting
\[
\scrK\defeq\{ F \in 2^B_+ \mid \text{there exists } H \in \Sigma \text{ with } H \preceq F\}.
\]
We need to prove that $\scrK$ is an ultracontact and $\sigma(\scrK) = \Sigma$. Let $F, G \subseteq B$.

\smallskip

\eqref{K0} holds for  $\scrK$ by the definition.

\smallskip

\eqref{K1} For the sake of contradiction, suppose that $F \in \scrK$ and $\zero \in F$. Then there exists $H \in \Sigma$ with $H \preceq F$. So, $\zero \in H$, which is not possible because $H \subseteq \Atom(B)$.

\smallskip

\eqref{K2} Let $x \in \Bne$. Since $B$ is finite, there exists $a \in \Atom(B)$ such that $a \leq x$. So, $\{a\} \preceq \{x\}$ and $\{a\} \in \Sigma$ because $\Sigma$ contains all singletons of $\Atom(B)$. The definition of $\scrK$ implies that $\{x\} \in \scrK$.

\smallskip

\eqref{K3} Assume that $F \preceq G \neq \emptyset$ and $F \in \scrK$. Since $F \in \scrK$, there exists $H \in \Sigma$ with $H \preceq F$. Since $H \preceq F$ and $F \preceq G$, Proposition~\ref{prop:support-very-basic}(4) yields that $H \preceq G$. Hence, $G\in\scrK$.

\smallskip

\eqref{K4} We prove this axiom by contraposition assuming that $F,G \notin \scrK$. Let $H \in \Sigma$. Since  $F,G \notin \scrK$, the definition of $\scrK$ implies that $H \npreceq F$ and $H \npreceq G$. Therefore, there exist $f \in F$ and $g \in G$ such that $a \nleq f$ and $a \nleq g$ for every $a \in H$. Since $H \subseteq \Atom(B)$, we obtain that $a \nleq f + g$ for every $a \in H$. We conclude that $H \npreceq F+G$. Since $H$ is an arbitrary element of $\Sigma$, it follows that $F+G \notin \scrK$.

\smallskip

Thus, we know that $\scrK$ is an ultracontact on $B$. It remains to prove that $\sigma(\scrK)=\Sigma$. As $H \preceq H$ for every $H \in \Sigma$, the inclusion from right to left is immediate. To prove the other inclusion, consider $F \in\sigma(\scrK)\defeq \scrK \cap 2^{\Atom(B)}_+$. Then $F \in 2^{\Atom(B)}_+$ and there exists $H \in \Sigma$ such that $H \preceq F$. So, for every $f \in F$ there exists $h \in H$ such that $h \leq f$. Since $H,F \subseteq \Atom(B)$, we have that $h \leq f$ implies $h=f$ for every $h \in H$ and $f \in F$. Thus, $F \subseteq H$. As $\Sigma$ is closed under non-empty subsets, we obtain that $F \in \Sigma$. This establishes the desired inclusion. Hence, $\sigma(\scrK) = \Sigma$.

This ends the proof of Theorem~\ref{thm:correspondence ultracontact finite}.
\end{proof}

\begin{corollary}
Let $B$ be a finite Boolean algebra. If two ultracontacts $\scrK_1$ and $\scrK_2$ on $B$ coincide on $2^{\Atom(B)}$, then $\scrK_1 = \scrK_2$.
\end{corollary}

\begin{lemma}\label{lem:contacts and atoms}
Let $\klam{B,\con}$ be a finite Boolean contact algebra and $x,y \in B$. Then $x \con y$ iff there exist $a,b \in \Atom(B)$ such that $\{a,b\} \preceq \{x,y\}$ and $a \con b$.
\end{lemma}
\begin{proof}
    Since the algebra is finite, for any $x,y\in\Bne$  there are atoms $a_1,\ldots,a_k$ and $b_1,\ldots,b_n$ such that $x=a_1+\ldots+a_k$ and $y=b_1+\ldots+b_n$. So, if $x\con y$, then $a_1+\ldots +a_k\con b_1+\ldots+b_n$. By \eqref{C2} and \eqref{C4} we may reduce the left-hand side to an $a_i$ with $1\leqslant i\leqslant k$, and the right-hand side to $b_j$ with $1\leqslant j\leqslant n$ such that $a_i\con b_j$. Clearly, $\{a_i,b_j\}\preceq\{x,y\}$. The other direction is a direct consequence of \eqref{C2} and \eqref{C3}.
\end{proof}

\begin{proposition}\label{prop:ultracontact and contacts finite}
If $\klam{B,\scrK}$ is a finite ultracontact algebra, then:
\begin{enumerate}[label=(\arabic*),ref=\arabic*,itemsep=0pt]
\item\label{prop:ultracontact and contacts finite:1} $x \Kon y$ iff there exist $a,b \in \Atom(B)$ such that $\{a,b \} \preceq \{x,y\}$ and $\{a, b \} \in\sigma(\scrK)$.
\item\label{prop:ultracontact and contacts finite:2} Let $\con$ be a contact relation on $B$. Then $\con = \Kon$ iff for all $a,b \in \Atom(B)$ we have
\[
a \con b \iff \{a,b\} \in \sigma(\scrK).
\]
\item\label{prop:ultracontact and contacts finite:3} For all $\scrK'\in\UC(B)$, $\Kon = \con_{\scrK'}$ iff $\sigma(\scrK)$ and $\sigma(\scrK')$ contain the same $2$-element sets.
\end{enumerate}
\end{proposition}

\begin{proof}
\eqref{prop:ultracontact and contacts finite:1} Since $\con_\scrK$ is a contact, Lemma~\ref{lem:contacts and atoms} implies that $x \Kon y$ iff there exist $a,b \in \Atom(B)$ such that $\{a,b\} \preceq \{x,y\}$ and $a \Kon b$. By the definitions of $\Kon$ and $\sigma(\scrK)$, the latter condition is equivalent to the existence of $a,b \in \Atom(B)$ such that $\{a,b \} \preceq \{x,y\}$ and $\{a, b \} \in \sigma(\scrK)$.

\smallskip

\eqref{prop:ultracontact and contacts finite:2} Suppose that $\con = \Kon$ and let $a,b \in \Atom(B)$. Then $a \con b$ iff $a \Kon b$. The definitions of $\Kon$ and $\sigma(\scrK)$ imply that $a \Kon b$ iff $\{a,b\} \in \sigma(\scrK)$. So, $a \con b$ iff $\{a,b\} \in \sigma(\scrK)$.
The converse implication is an immediate consequence of \eqref{prop:ultracontact and contacts finite:1} and Lemma~\ref{lem:contacts and atoms}.

\smallskip

\eqref{prop:ultracontact and contacts finite:3} It follows from \eqref{prop:ultracontact and contacts finite:1} that $\Kon = \con_{\scrK'}$ iff $\sigma(\scrK)$ and $\sigma(\scrK')$ contain the same sets of cardinality at most $2$. Since both $\sigma(\scrK)$ and $\sigma(\scrK')$ contain all the singletons in $2^{\Atom(B)}_+$, we conclude that \eqref{prop:ultracontact and contacts finite:3} holds.
\end{proof}

\begin{remark}\label{rem:geometrical-meaning-1}
    Let us rephrase the points above in terms of simplicial complexes to emphasize the geometrical meaning of the proposition. If $\klam{B,\scrK}$ is a finite UCA, then we turn our attention to both the contact algebra $\frB\defeq\klam{B,\con_{\scrK}}$ and the simplicial complex $\frS\defeq\klam{\Atom(B),\sigma(\scrK)}$. From this point of view, we can express the properties above in the following way:
\begin{enumerate}[label=(\arabic*),ref=\arabic*,itemsep=0pt]
\item Regions $x$ and $y$ of $\frB$ are in contact iff they are above (in the sense of the Boolean order) vertices of a simplex of $\frS$; in particular, two atoms of the algebra are in contact if they are vertices of a simplex.
\item $\con_{\scrK}$ is the unique contact on $B$ such that the pairs of atoms in $\con_{\scrK}$ are precisely the pairs of vertices of simplices of $\frS$.
\item An ultracontact $\scrK'$ gives rise to the same contact as $\scrK$ iff the simplicial complex $\frS'\defeq\klam{\Atom(B),\sigma(\scrK')}$ agrees with $\frS$ on $2$-element sets (which is the case when the underlying graphs of the two complexes coincide).\QED
\end{enumerate}
\end{remark}

\begin{theorem}\label{thm:correspondence contact finite}
Let $\klam{B,\con}$ be a finite Boolean contact algebra. Then there exist a smallest and a greatest ultracontact on $B$ whose corresponding contact is $\con$. In particular, if $\scrK$ is an ultracontact on $B$, then:
\begin{enumerate}
\item $\scrK$ is the smallest ultracontact on $B$ such that $\con = \Kon$ iff
\[
\sigma(\scrK)= \{ \{a,b \} \in 2^{\Atom(B)}_+ \mid a \con b \}.
\]
\item $\scrK$ is the greatest ultracontact on $B$ such that $\con = \Kon$ iff
\[
\sigma(\scrK) = \{ H \in 2^{\Atom(B)}_+ \mid a \con b \text{ for all } a,b \in H\}.
\]
\end{enumerate}
\end{theorem}
\begin{proof}
By Proposition~\ref{prop:ultracontact and contacts finite}\eqref{prop:ultracontact and contacts finite:2}, we have that $\con = \Kon$ exactly when $a \con b$ iff $\{a,b\} \in \sigma(\scrK)$, for all $a,b \in \Atom(B)$. It is straightforward to verify that the set $\{ \{a,b \} \in 2^{\Atom(B)}_+ \mid a \con b \}$ is the smallest $\Sigma \in \SC(\Atom(B))$ such that $a \con b$ iff $\{a,b\} \in \Sigma$, for all $a,b \in \Atom(B)$. By Theorem~\ref{thm:correspondence ultracontact finite}, $\sigma$ is an order isomorphism between $\UC(B)$ and $\SC(\Atom(B))$. Therefore, $\scrK$ is the smallest ultracontact on $B$ such that $\con = \Kon$ iff
\[
\sigma(\scrK)= \{ \{a,b \} \in 2^{\Atom(B)}_+ \mid a \con b \}.
\]

It is also straightforward to verify that the set $\{ H \in 2^{\Atom(B)}_+ \mid a \con b \text{ for all } a,b \in H\}$ is the largest $\Sigma \in \SC(\Atom(B))$ such that $a \con b$ iff $\{a,b\} \in \Sigma$, for all $a,b \in \Atom(B)$. Thus, a similar argument yields that $\scrK$ is the greatest ultracontact on $B$ such that $\con = \Kon$ iff
\[
\sigma(\scrK) = \{ H \in 2^{\Atom(B)}_+ \mid a \con b \text{ for all } a,b \in H\}.\qedhere
\]
\end{proof}

\begin{remark}
As in the case of Remark~\ref{rem:geometrical-meaning-1}, let us take a geometrical perspective on the theorem above. Given a finite Boolean contact algebra $\klam{B,\con}$, the smallest ultracontact $\scrK$ for which $\con_{\scrK}$ is precisely $\con$ is the ultracontact whose simplicial complex is the graph with vertices $\Atom(B)$ and whose edges are all the pairs of atoms in $\con$.  On the other hand, the complex $\klam{\Atom(B),\sigma(\scrK')}$ corresponding to the largest ultracontact $\scrK'$ with the same property that $\con=\con_{\scrK'}$ consists of all $\con$-cliques on the set of atoms.

For a more concrete example, we bring the reader's attention to the two ultracontacts $\scrK$ and $\scrK_{a,b,c}$ considered in Example~\ref{ex:diff-UCs-same-C}. They are, respectively, the smallest and largest ultracontacts whose corresponding contact is the full contact on the  $8$-element Boolean algebra $B$. This can be checked easily using Theorem~\ref{thm:correspondence contact finite}.
Indeed, $\sigma(\scrK)$ consists of all non-empty subsets of $\Atom(B)$ that have at most $2$ elements, while $\sigma(\scrK_{a,b,c})$ is the collection of all non-empty subsets of $\Atom(B)$.
The simplicial complex corresponding to $\scrK$ is then a complete graph on $3$ vertices, and the one corresponding to $\scrK_{a,b,c}$ is a triangular face together with its edges and vertices (see Figure~\ref{fig:simplicial complexes}).\QED
\end{remark}
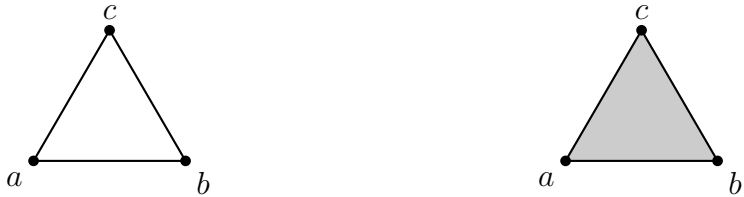
\begin{figure}[ht]
\centering
\begin{tikzpicture}[scale=2]

\begin{scope}
  \coordinate (A) at (0,0);
  \coordinate (B) at (1,0);
  \coordinate (C) at (0.5,{sqrt(3)/2});

  \draw[line width=0.8pt] (A) -- (B) -- (C) -- cycle;

  \fill (A) circle (1pt) node[below left] {$a$};
  \fill (B) circle (1pt) node[below right] {$b$};
  \fill (C) circle (1pt) node[above] {$c$};
\end{scope}

\begin{scope}[xshift=3.5cm]
  \coordinate (A') at (0,0);
  \coordinate (B') at (1,0);
  \coordinate (C') at (0.5,{sqrt(3)/2});

  \fill[gray!40] (A') -- (B') -- (C') -- cycle;
  \draw[line width=0.8pt] (A') -- (B') -- (C') -- cycle;

  \fill (A') circle (1pt) node[below left] {$a$};
  \fill (B') circle (1pt) node[below right] {$b$};
  \fill (C') circle (1pt) node[above] {$c$};
\end{scope}
\end{tikzpicture}
\caption{The simplicial complex corresponding to $\scrK$ on the left, and the one corresponding to $\scrK_{a,b,c}$ on the right}\label{fig:simplicial complexes}
\end{figure}

We conclude this section with the following observation.

\begin{corollary}
 Let $\klam{B,\scrK}$ be a finite ultracontact algebra.
 \begin{enumerate}[label=(\arabic*),ref=\arabic*,itemsep=0pt]
     \item If $\scrK=\Kmin$, then the only simplices in $\sigma(\scrK)$ are vertices (the singletons of atoms of $B$).
     \item If $\scrK=\Kmax$, then $\sigma(\scrK)=2^{\Atom(B)}_+$, and in this case $\klam{\sigma(\scrK),\subseteq}$ is order isomorphic to $\klam{\Bne,\leq}$.
 \end{enumerate}
\end{corollary}
\begin{proof}
    For (1), recall that $\Kmin$ contains all the sets that have a non-zero lower bound, and no collection of atoms that has at least two elements has such a bound.

    \smallskip

    For (2), it is sufficient to observe that $\Kmax$ contains all non-empty subsets of $B^{+}$, and so all non-empty sets of atoms. As is well known, in the case of a complete and atomic algebra, there is an isomorphic correspondence between $B$ and $2^{\Atom(B)}$, and so also between $\klam{\sigma(\scrK),\mathord{\subseteq}}$ and $\klam{B^{+},\mathord{\leq}}$.
\end{proof}

\section{Stack systems}\label{sec:stacks}

In this section, we introduce the notion of a \emph{stack system} and a \emph{stack algebra} that are definitionally equivalent to ultracontact and ultracontact algebras, and offer a~different perspective on the general phenomenon we study.

\begin{definition}\label{def:SS}
We say that $\scrS\subseteq \Stack(B)$ is a \emph{stack system} if the following axioms hold for all $U,V\in \Stack(B)$ and all $x\in B$:
\begin{enumerate}[label=(SS\arabic*),ref=SS\arabic*,itemsep=3pt]
\addtocounter{enumi}{-1}
\item \label{SS0} $\emptyset\notin \scrS$.
    \item \label{SS1} $B\notin \scrS$.
    \item \label{SS2} If $x\neq\zero$, then $\upop x\in \scrS$.
    \item \label{SS3} If $\emptyset \neq U\subseteq V$ and $V\in \scrS$, then $U\in \scrS$.
    \item \label{SS4} If $U\cap V\in \scrS$, then $U\in \scrS$ or $V\in \scrS$.
\end{enumerate}
A \emph{stack algebra} is a pair $\klam{B,\scrS}$ such that $B$ is a Boolean algebra and $\scrS$ is a stack system on $B$.\QED
\end{definition}

We aim to show that ultracontact algebras and stack algebras are definitionally equivalent structures.

\begin{lemma}\label{l: preceq and + for upsets 2}
For any UCA $\klam{B,\scrK}$, if $F\in \scrK$, then $\ua{F}\in \scrK$.
\end{lemma}
\begin{proof}
Let $F\in\scrK$, which in particular means that $F\neq\emptyset$. Then $\upop F\neq\emptyset$, and since $\upop F\subseteq\upop F$, by Proposition~\ref{prop:support-very-basic}(6), $F\preceq \upop F$. Then, by \eqref{K3} we obtain that $\upop F\in\scrK$.
\end{proof}

For an ultracontact $\scrK$, define
\begin{equation}\tag{$\mathrm{df}\,\SK$}
 \SK \defeq  \{\upop F\mid F\in \scrK\},
\end{equation}
and for a stack system $\scrS$ we put
\begin{equation}\tag{$\mathrm{df}\,\KS$}
 \KS \defeq \{F\subseteq  B\mid \upop F\in \scrS\}.
\end{equation}

\begin{proposition}\label{prop: Ku subset of K}
    For any Boolean algebra $B$,
    \begin{enumerate}
        \item $\SK\subseteq \scrK$ for every ultracontact $\scrK$.
        \item $\scrS\subseteq \KS$ for every stack system $\scrS$.
    \end{enumerate}
\end{proposition}
\begin{proof}
(1) If $\scrK$ is an ultracontact, and $\ua{F}\in\SK$, for some $F\in\scrK$, then $\upop F\in\scrK$ by Lemma~\ref{l: preceq and + for upsets 2}.

\smallskip

(2) Let $F\in\scrS$. Then, $F=\ua{F}$, since $F$ is a stack. So $F\in\KS$.
\end{proof}

\begin{lemma}\label{lem:for-last-lemma}
    For any $F,G\in 2^B$, $\upop(F+G)=\upop F\cap\upop G$.
\end{lemma}
\begin{proof}
    $(\subseteq)$ If $x\in \upop(F+G)$, then there are $f\in F$ and $g\in G$ such that $f+g\leq x$. Thus $f\leq x$ and $g\leq x$, and so $x\in\upop F\cap\upop G$.

    \smallskip

    $(\supseteq)$ If $x\in \upop F\cap\upop G$, then there are $f\in F$ and $g\in G$ such that $f\leq x$ and $g\leq x$, and so $f+g\leq x$. Thus $x\in\upop(F+G)$.
\end{proof}

\begin{lemma}\label{lem:SK-and-KS-operations}
     Let $B$ be a Boolean algebra.
     \begin{enumerate}
         \item $\SK$ is a~stack system for every ultracontact $\scrK$.
         \item $\KS$ is an ultracontact for every stack system $\scrS$.
     \end{enumerate}
 \end{lemma}
\begin{proof}
(1) Suppose that $\scrK$ is an ultracontact. We will show that $\SK$ is a stack system.

\smallskip

\eqref{SS0} If $\emptyset\in\SK$, then for some $F\in\scrK$, $\emptyset=\upop F$. But this means that $F=\emptyset$, which contradicts \eqref{K0}.

\smallskip

\eqref{SS1}  If $\ua{F}=B$ for some $F\in \scrK$, then $\zero \in F$, contradicting \eqref{K1}.

\smallskip

\eqref{SS2} If $x \neq \zero$, then $\set{x} \in \scrK$ by \eqref{K2}, and thus $\ua{x}\in \SK$.

\smallskip

\eqref{SS3} Let $U,V \in \Stack(B), V \in \SK$, and $\emptyset \neq U \subseteq V$. By definition of $\SK$ there is some $F \in \scrK$ such that $\ua{F} = V$. By Proposition~\ref{prop:support-very-basic}(6) $F \preceq U$, and \eqref{K3} implies $U \in \scrK$.  Now, $U = \ua{U}$ shows that $U \in \SK$.

\smallskip

\eqref{SS4} Let $U,V \in \Stack(B)$ and $U \cap V \in \SK$. By Lemma~\ref{lem:intersection-for-stacks}, $U + V \in \SK$, so there is an $M\in\scrK$ such that $U+V=\upop M$, and in consequence $M\preceq U+V$ by Proposition~\ref{prop:support-very-basic}(6). By \eqref{K3}, $U+V\in\scrK$, and by \eqref{K4}, either $U$ or $V$ is in $\scrK$. If it is $U$, then $\ua{U}\in\SK$, and since $U = \ua{U}$, we obtain $U\in\SK$.  If it is $V$, then similarly, $V\in\SK$.

\smallskip

(2) Suppose that $\scrS$ is stack system. We will show that $\KS$ is an ultracontact.

  \smallskip

    \eqref{K0} Since $\emptyset\not\in\scrS$ by \eqref{SS0}, we also have $\emptyset\notin \KS$.

\smallskip

    \eqref{K1} Suppose that $\zero\in F \subseteq B$. Then $\ua{F} = B \not\in \scrS$ which implies $F \not\in \KS$.

\smallskip

    \eqref{K2} If $x \neq \zero$, then $\ua{x} \in \scrS$ by \eqref{SS2} which implies $\set{x} \in \KS$.

\smallskip

    \eqref{K3} Let $F \in \KS, G \neq \emptyset$, and $F \preceq G$. By Proposition~\ref{prop:support-very-basic}(6) we have $G \subseteq \ua{F}$. Then, $\ua{G} \subseteq \ua{F}$, and $\ua{F} \in \scrS$ since $F \in \KS$. Now, \eqref{SS3} implies that $\ua{G} \in \scrS$, and therefore, $G \in \KS$.

\smallskip

    \eqref{K4} Let $F + G \in \KS$. Then, $\ua{(F+G)} \in \scrS$ by definition of $\KS$, and therefore, $\ua{F} \cap \ua{G} \in \scrS$ by Lemma~\ref{lem:for-last-lemma}. It follows from \eqref{SS4} that either $\ua{F} \in \scrS$ which implies $F \in \KS$, or $\ua{G} \in \scrS$ which implies $G\in \KS$
\end{proof}

\begin{theorem}\label{th:UC-one-to-one-SS}
The mappings $\scrK\mapsto\SK$ and $\scrS\mapsto\KS$ are mutual inverses, that is $\scrK=\scrK_{\SK}$ and $\scrS=\scrS_{\KS}$, so for any Boolean algebra $B$, there is a one-to-one correspondence between ultracontacts and stack systems on~$B$.
\end{theorem}
\begin{proof}
First, let us check $\scrK=\scrK_{\SK}$ for every ultracontact $\scrK$. The inclusion $\scrK\subseteq \scrK_{\SK}$ is clear, since if $F\in \scrK$ then $\ua{F}\in \SK$, by definition, hence $F\in \scrK_{\SK}$. For the other inclusion, suppose that $F\subseteq B$ is such that $\ua{F}\in \SK$. Note that $F\neq \emptyset$ as $\ua{\emptyset}=\emptyset$ and $\emptyset\notin \scrS_{\scrK}$ by \eqref{SS0}. This means that there is $G\in \scrK$ with $\ua{F}=\ua{G}$. In particular, $F\subseteq \ua{G}$, that is $G\preceq F$ by Proposition~\ref{prop:support-very-basic}(6), and so $F\in \scrK$ by \eqref{K3}.

\smallskip

Let us now check that $\scrS=\scrS_{\KS}$, for all stack systems $\scrS$. If $U\in \scrS$ then $U\in \KS$ by Proposition \ref{prop: Ku subset of K}(2),  and it follows that $\ua{U}=U\in \scrS_{\KS}$. Conversely, let $U\in \scrS_{\KS}$. Then $U=\ua{F}$ for some $F\in \KS$. As $F\in \KS$, it follows that $\ua{F}=U\in \scrS$.
\end{proof}

Given a Boolean algebra $B$, let $\STS(B)$ be the poset of all stack systems on $B$ ordered by the set-theoretical inclusion. From the definitions, for any ultracontacts $\scrK_1$, $\scrK_2$, and any stack systems $\scrS_1$, $\scrS_2$ the following equivalences hold
\begin{align*}
    \scrK_1\subseteq\scrK_2\iff\scrS_{\scrK_{1}}\subseteq\scrS_{\scrK_{2}},\\
    \scrS_1\subseteq\scrS_2\iff\scrK_{\scrS_{1}}\subseteq\scrK_{\scrS_{2}}.
\end{align*}
Moreover, by Theorem~\ref{th:UC-one-to-one-SS}, we know that there is a~bijective correspondence between ultracontacts and stack systems on $B$. This implies that $\STS(B)$ is a complete lattice. In particular, $\scrS_{\Kmin}$ and $\scrS_{\Kmax}$ are the zero $\Smin$ and the unity $\Smax$ of $\STS(B)$. We can describe explicitly what they are.

\begin{proposition} For every Boolean algebra $B$
\begin{align*}
\Smin&{}=\{M\in \Stack(B)\mid \emptyset\neq M\subseteq \ua{x},\text{ for some }x\in B^+\},\\
\Smax&{}=\Stack(B)\setminus\{\emptyset,B\}
\end{align*}
are---respectively---the zero and the unity of $\STS(B)$.
\end{proposition}
\begin{proof}
We have that
\[
\Kmin=\{M\subseteq B\mid\emptyset\neq M\subseteq\upop x,\text{ for some }x\in B^+\}.
\]
By the definition, $\scrS_{\Kmin}\defeq\{\upop F\mid F\in\Kmin\}$. Consequently,
\[
\scrS_{\Kmin}=\{M\in \Stack(B)\mid \emptyset\neq M\subseteq \ua{x},\text{ for some }x\in B^+\}.
\]
So, the first equality follows by the order isomorphisms of the lattices of stack systems and ultracontacts. For the second one, recall that $\Kmax=2^{B}_{+}$, so by the definition of $\scrS_{\Kmax}$ we obtain that it consists of all non-empty and proper stacks, and again, the second equality from the proposition follows from the order isomorphism of the two lattices.
\end{proof}

\section{Future work}\label{sec:future}

After finishing the work on the final version of this paper, we discovered that similar notions to the ultracontact, but formulated in a point-based setting, were introduced in \cite{Katetov-OCSASOM}  under the name of \emph{merotopy} (or \emph{merotopic structure}) and in \cite{Herrlich-ACON} under the name of \emph{nearness}. In \cite{Herrlich-TS}, the author proved that the two concepts are equivalent. The ultracontact notion of this paper can be seen as an abstract, point-free approach to merotopies/nearnesses in a similar way that contact algebras can be considered a point-free approach to proximity spaces (of whose merotopies and nearnesses are generalizations).  It would be interesting to see whether any results obtained via merotopies could be generalized in the ultracontact framework. This, however, goes way beyond the objectives of this paper.

Already in Footnote~\ref{foot:ECA} we pointed to extended contact algebras and sequent algebras as structures motivated by region-based reasoning that can be compared to ultracontact algebras. We postpone the comparison to future investigations since, for now, we cannot see a way to relate them.

The most natural continuation of the work presented above concerns finding a~proper topological representation of ultracontact algebras in the spirit of \cite{Duntsch-et-al-RTBCA} and \cite{Dimov-et-al-CARBTSPA1} for Boolean contact algebras.

Another interesting problem involves studying ultracontact in larger classes of algebraic structures, such as p-algebras, Heyting algebras, or distributive lattices.

Also, investigating the deeper connections between ultracontacts and stack systems is a worthwhile direction for future research.

\section*{AI Disclosure}
All the concepts, results, and examples in the paper were created, proven, or obtained by the authors without any help of AI tools. Perplexity Pro was used to review the final version of the manuscript for language issues, typos, and minor mistakes.

\end{document}